\documentclass[11pt]{article}

\usepackage{amsmath,amssymb}
\usepackage{latexsym}
\usepackage{graphicx}
\usepackage{bm}

\newtheorem{thm}{Theorem}
\newtheorem{cor}[thm]{Corollary}

\newtheorem{lem}[thm]{Lemma}

\newtheorem{rem}{Remark}

\newenvironment{proof}{\begin{trivlist}
                       \item[]{\bf Proof.}
                       \hspace{0cm}}{\hfill $\Box$
                       \end{trivlist}}

\def\Re{\mathop{\rm Re}}

\begin{document}
\title{The Dynamical Systems Method for solving nonlinear equations with monotone operators}

\author{ N. S. Hoang$\dag$\footnotemark[1]\, and\,  
A. G. Ramm$\dag$\footnotemark[3]
\\
\\
$\dag$Mathematics Department, Kansas State University,\\
Manhattan, KS 66506-2602, USA
}

\renewcommand{\thefootnote}{\fnsymbol{footnote}}
\footnotetext[3]{Corresponding author. Email: ramm@math.ksu.edu}
\footnotetext[1]{Email: nguyenhs@math.ksu.edu}

\date{}
\maketitle

\begin{abstract} \noindent A review of the authors's results is given.  
Several methods are discussed for solving nonlinear equations $F(u)=f$, 
where $F$ is a
monotone operator in a Hilbert space, and noisy data are given in place
of the exact data.
A discrepancy
principle for solving the equation is formulated and justified.  Various
versions of the Dynamical Systems Method (DSM) for solving the equation 
are
formulated.  These methods consist of a regularized Newton-type method, a
gradient-type method, and a simple iteration method.  A priori and a
posteriori choices of stopping rules for these methods are proposed and
justified. Convergence of the solutions, obtained by these methods, to
the minimal norm solution to the equation $F(u)=f$ is proved. Iterative
schemes with a posteriori choices of stopping rule corresponding to the
proposed DSM are formulated. Convergence of these iterative schemes to a
solution to equation $F(u)=f$ is justified.  
New nonlinear differential inequalities are derived and applied
to a study of large-time behavior of solutions to evolution equations.
Discrete versions of these inequalities are established.


\end{abstract}

\section{Introduction}

Consider equation
\begin{equation}
\label{eq1}
F(u) = f,
\end{equation}
where $F$ is an operator in a Hilbert space $H$.
Throughout this paper we assume that $F$ is a monotone continuous operator.
Monotonicity is understood as follows:
\begin{equation}
\label{eq2}
\langle F(u) - F(v), u-v\rangle \ge 0,\qquad \forall u,v\in H.
\end{equation}
We assume that equation \eqref{eq1} has a solution, possibly non-unique. 
Assume that $f$ is not known but $f_\delta$, the "noisy data",  
$\|f_\delta-f\|\le \delta$, are known. 

There are many practically important problems which are ill-posed in the 
sense of J.Hadamard. 
Problem \eqref{eq1} is well-posed in the sense of Hadamard if and only 
if(=iff) $F$ is injective, surjective, and the inverse operator
$F^{-1}$ is continuous. 
To solve ill-posed problem \eqref{eq1}, one has to use regularization 
methods rather than
the classical Newton's or Newton-Kantorovich's methods. 
Regularization methods for stable solution of linear ill-posed problems 
have been studied 
extensively (see \cite{M}, \cite{R499}, \cite{VA} and references therein). 
Among regularization methods, the Variational Regularization (VR) 
is one of the frequently used methods. 
When $F=A$, where $A$ is a linear operator, the VR method consists 
of minimizing the following functional:
\begin{equation}
\label{eq4}
\|Au-f_\delta\|^2 + \alpha \|u\|^2\to \min.
\end{equation}
The minimizer $u_{\delta,a}$ of problem \eqref{eq4} can be found from the Euler equation:
\begin{equation*}
(A^*A+\alpha I)u_{\delta,\alpha} = A^*f_\delta.
\end{equation*}
In the VR method the choice of the regularization parameter $\alpha$ is important.
Various choices of the regularization parameter have been proposed and 
justified.
Among these, the discrepancy principle (DP) appears to be
the most efficient in practice (see \cite{M}).
According to the DP one chooses $\alpha$ as the solution to the following 
equation:
\begin{equation}
\label{eq3}
\|A u_{\delta,\alpha} - f_\delta\| = C\delta,\qquad 1<C=const.
\end{equation}

When the operator $F$ is nonlinear, the theory is less complete
(see \cite{A}, \cite{TLY}). 
In this case, one may try to minimize the functional
\begin{equation}
\label{eq5}
\|F(u)-f_\delta\|^2 + \alpha \|u\|^2\to \min.
\end{equation}
as in the case of linear operator F. The 
minimizer to problem \eqref{eq5} solves the following 
Euler equation
\begin{equation}
\label{eq*}
F'(u_{\delta,\alpha})^*F(u_{\delta,\alpha})+\alpha u_{\delta,\alpha} = 
F'(u_{\delta,\alpha})^*f_\delta.
\end{equation}
However, there are several principal difficulties in nonlinear
problems: there are no general results concerning 
the  solvability of \eqref{eq*}, and the notion 
of minimal-norm solution does not make sense, in general, when $F$ is
nonlinear. 
Other methods for solving \eqref{eq1} with nonlinear $F$ have been 
studied. Convergence proofs of these methods often rely on the
source-type assumptions. These assumptions are difficult to verify 
in practice and they may not hold. 

Equation  \eqref{eq1} with a monotone operator $F$ is 
of interest and importance in many applications. 
Every solvable linear operator equation $Au=f$ can be reduced
to solving operator equation with a monotone operator $A^*A$.
For equations with a bounded operator $A$ this is a simple fact,
and for unbounded, closed, densely defined linear operators $A$
it is proved in \cite{R500}, \cite{R504}, \cite{R522}, \cite{R499}.

Physical problems with dissipation of energy often
can be reduced to solving equations with monotone operators \cite{R118}.
For simplicity we present the results for equations in
Hilbert space, but some results can be generalized to the
operators in Banach spaces.

When $F$ is monotone then the notion minimal-norm solution makes sense 
(see, e.g., \cite[p. 110]{R499}). 
In \cite{T}, Tautenhahn
studied a discrepancy principle for solving equation \eqref{eq1}. 
The discrepancy principle 
in \cite{T} requires solving for $\alpha$ the following equation:
\begin{equation}
\label{eq6}
\|(F'(u_{\delta,\alpha})+\alpha I)^{-1}(F(u_{\delta,\alpha})-f_\delta)\|=C\delta,\qquad 1<C=const, 
\end{equation}
where $u_{\delta,\alpha}$ solves the equation:
$$
F(u_{\delta,\alpha}) + \alpha u_{\delta,\alpha} = f_\delta.
$$
For this discrepancy principle optimal rate of convergence is obtained 
in \cite{T}. 
However, the convergence of the method is justified under 
source-type assumptions and other restrictive assumptions. 
These assumptions often do not hold and some of them cannot
be verified, in general. In addition, equation \eqref{eq6} is difficult 
to solve numerically. 

A continuous analog of the Newton method for solving well-posed operator equations
was proposed in \cite{G}, in 1958.
In \cite{R401}, \cite{R538}--\cite{R529}, and in the monograph \cite{R499} 
the Dynamical Systems Method for solving operator equations is studied
systematically. 
The DSM consists of finding a nonlinear map $\Phi(t,u)$ such that the 
Cauchy problem
\begin{equation}
\label{eq7}
\dot{u}=\Phi(t,u),\qquad u(0)=u_0,
\end{equation}
has a unique solution for all $t\ge0$, there exists $\lim_{t\to\infty}u(t):=u(\infty)$,
and $F(u(\infty))=f$,
\begin{equation}
\label{eq8}
\exists !\, \,u(t)\quad \forall t\ge 0;\qquad \exists u(\infty);\qquad F(u(\infty))=f.
\end{equation}
Various choices of $\Phi$ were proposed in \cite{R499} for \eqref{eq8} to hold.
Each such choice yields a version of the DSM.

In this paper, several methods developed by the authors for solving 
stably equation \eqref{eq1} with a monotone operator $F$ in a Hilbert 
space and noisy data $f_\delta$, given  in place of the exact data $f$, 
are 
presented. A discrepancy principle (DP) is formulated for solving stably 
equation \eqref{eq1} is formulated and justified. In this DP the only 
assumptions on $F$ are the continuity and monotonicity. Thus, our result 
is quite general and can be applied for a wide range of problems. 
Several versions of the Dynamical Systems Method (DSM) for solving 
equation \eqref{eq1} are formulated. These versions of the DSM are 
Newton-type method, 
gradient-type method and a simple iterations method. A priori and a 
posteriori choices of stopping rules for several versions of the DSM and 
for the corresponding iterative 
schemes are proposed and justified. Convergence of the solutions of 
these versions of the DSM to the minimal-norm solution to the equation 
$F(u)=f$ is
proved. Iterative schemes, corresponding to the proposed versions of the 
DSM, are 
formulated. Convergence of these iterative schemes to a solution to 
equation $F(u)=f$ is established. When one uses these iterative schemes 
one does not have to solve a nonlinear equation for the regularization 
parameter. The stopping time is chosen automatically in the course of 
calculations. Implementation of these methods is illustrated in Section 6 
by a numerical experiment. In Sections 2 and 3 basic and auxiliary results 
are formulated, in Section 4 proofs are given, in Section 5 ideas of
application of the basic nonlinear inequality (94) are outlined.

\section{Basic results}

\subsection{A discrepancy principle}

Let us consider the following equation
\begin{equation}
\label{eq9}
F(V_{\delta,a})+aV_{\delta,a}-f_\delta = 0,\qquad a>0,
\end{equation}
where $a=const$. It is known (see, e.g., \cite[p.111]{R499}) 
that equation \eqref{eq9} with a monotone continuous operator $F$ has 
a unique solution for any $f_\delta\in H$. 

Assume that equation \eqref{eq1} has a solution. 
It is known that the set of solution $\mathcal{N}:=\{u:F(u)=f\}$ 
is convex and closed if $F$ is monotone 
and continuous (see, e.g., \cite{R499}, p.110). A closed and convex set 
$\mathcal{N}$ in $H$ has a unique minimal-norm element.
This minimal-norm solution to \eqref{eq1} is denoted by $y$.

\begin{thm}
\label{thm1}
Let $\gamma\in (0,1]$ and $C>0$ be some constants such that $C\delta^\gamma>\delta$.
Assume that 
$\|F(0)-f_\delta\|>C\delta^\gamma$. Let $y$ be the minimal-norm solution to equation \eqref{eq1}.
Then there exists a unique $a(\delta)>0$ such that
\begin{equation}
\label{eq10}
\|F(V_{\delta,a(\delta)}) - f_\delta \| = C\delta^\gamma,
\end{equation}
where $V_{\delta, a(\delta)}$ solves \eqref{eq3} with $a = a(\delta)$.

If $0< \gamma <1$ then
\begin{equation}
\label{eq11}
\lim_{\delta\to 0} \|V_{\delta, a(\delta)}-y\| = 0.
\end{equation}
\end{thm}

Instead of using \eqref{eq9}, one may use the following equation:
\begin{equation}
\label{eq20}
F(V_{\delta,a})+a(V_{\delta,a}-\bar{u})-f_\delta = 0,\qquad a>0,
\end{equation}
where $\bar{u}$ is an element of $H$. Denote $F_1(u):=F(u+\bar{u})$.
 Then $F_1$ is monotone and continuous. 
Equation \eqref{eq20} can be written as:
\begin{equation}
\label{eq21}
F_1(U_{\delta,a})+aU_{\delta,a}-f_\delta = 0,
\qquad U_{\delta,a}:=V_{\delta,a}-\bar{u},\quad a>0.
\end{equation}
Applying Theorem~\ref{thm1} with $F=F_1$ one gets the following result:

\begin{cor}
Let $\gamma\in (0,1]$ and $C>0$ be some constants such that $C\delta^\gamma>\delta$.
Let $\bar{u}\in H$ and $z$ be the solution to \eqref{eq1} with 
minimal distance to $\bar{u}$.
Assume that 
$\|F(\bar{u})-f_\delta\|>C\delta^\gamma$. 
Then there exists a unique $a(\delta)>0$ such that
\begin{equation}
\label{eq22}
\|F(\tilde{V}_{\delta,a(\delta)}) - f_\delta \| = C\delta^\gamma,
\end{equation}
where $\tilde{V}_{\delta, a(\delta)}$ solves the following equation:
\begin{equation}
\label{eq23}
F(\tilde{V}_{\delta,a})+a(\delta)(\tilde{V}_{\delta,a}-\bar{u})-f_\delta = 0.
\end{equation}

If $\gamma \in (0,1)$ then this $a(\delta)$ satisfies
\begin{equation}
\label{eq24}
\lim_{\delta\to 0} \|\tilde{V}_{\delta, a(\delta)}-z\| = 0.
\end{equation}
\end{cor}

The following result is useful for the implementation of our DP. 

\begin{thm}
\label{theorem2}
Let $\delta, F, f_\delta$, and $y$ be as in Theorem~\ref{thm1} and $0<\gamma<1$. 
Assume that $v_\delta\in H$ and $\alpha(\delta)>0$ 
satisfy the following conditions:
\begin{equation}
\label{eq25}
\|F(v_\delta)+\alpha(\delta) v_\delta - f_\delta\| \le \theta \delta,\qquad \theta>0,
\end{equation}
and
\begin{equation}
\label{eq26}
C_1\delta^\gamma \le \|F(v_\delta) - f_\delta\| \le C_2 \delta^\gamma,
\qquad
0< C_1 < C_2.
\end{equation}
Then one has:
\begin{equation}
\label{eq27}
\lim_{\delta\to 0}\|v_\delta - y\| = 0.
\end{equation}
\end{thm}

\begin{rem}{\rm
Based on Theorem~\ref{theorem2}
an algorithm for solving nonlinear equations with monotone Lipchitz continuous operators is outlined in \cite{R554}.
}
\end{rem}

\begin{rem}{\rm
It is an open problem to choose $\gamma$ and $C$ optimal in some 
sense.
}
\end{rem}

\begin{rem}{\rm Theorem~\ref{thm1} and Theorem~\ref{theorem2} do not 
hold, in general, for $\gamma=1$. Indeed, let $Fu = \langle u,p\rangle 
p$, $\|p\|=1,\, p\perp \mathcal{N}(F):=\{u\in H: Fu=0\}$, $f=p$, 
$f_\delta=p+q\delta$, where $\langle p,q\rangle=0$, $\|q\|=1$, $Fq=0$, 
$\|q\delta\|=\delta$. One has $Fy=p$, where $y=p$, is the minimal-norm 
solution to the equation $Fu=p$. Equation $Fu+au=p+q\delta$, has the 
unique solution $V_{\delta,a}=q\delta/a + p/(1+a)$. Equation 
\eqref{eq10} is $C\delta=\|q\delta + (ap)/(1+a)\|$. This equation yields 
$a=a(\delta)=c\delta/(1-c\delta)$, where $c:=(C^2-1)^{1/2}$, and we 
assume $c\delta<1$. Thus, 
$\lim_{\delta\to 0}V_{\delta,a(\delta)}=p+c^{-1}q:=v$, and $Fv=p$. 
Therefore $v=\lim_{\delta\to0}V_{\delta,a(\delta)}$ is not $p$, i.e., is 
not the minimal-norm solution to the equation $Fu=p$. This argument is 
borrowed from \cite[p. 29]{R470}.

If  equation  \eqref{eq1} has a unique solution and  $\gamma=1$, then 
one can prove convergence  \eqref{eq11} and  \eqref{eq27}.
}
\end{rem}

\subsection{The Dynamical Systems Method}

Let $a(t)\searrow 0$ be a positive and strictly decreasing sequence. 
Let $V_\delta(t)$ solve the following equation: 
\begin{equation}
\label{eq41.2}
F(V_\delta(t))+a(t)V_\delta(t) - f_\delta = 0. 
\end{equation}

Throughout the paper we assume that equation $F(u)=f$ has a solution in $H$, possibly nonunique,
and $y$ is the minimal-norm solution to this equation.
Let $f$ be unknown but $f_\delta$ be given, $\|f_\delta-f\|\le \delta$.

\subsubsection{The  Newton-type DSM}

Denote
\begin{equation}
\label{eq42}
A:=F'(u_\delta(t)),\quad A_a:=A + aI,
\end{equation}
where $I$ is the identity operator, 
and $u_\delta(t)$ solves the following Cauchy problem:
\begin{equation}
\label{eq43}
\dot{u}_\delta = -A_{a(t)}^{-1}[F(u_\delta)+a(t)u_\delta-f_\delta],
\quad u_\delta(0)=u_0.
\end{equation}
We assume below that $||F(u_0)-f_\delta||>C_1\delta^{\zeta}$,
where $C_1>1$ and $\zeta\in (0,1]$ are some constants.
We also assume without loss of generality that $\delta\in (0,1)$.
Assume that equation \eqref{eq1} has a solution, possibly nonunique,
and $y$ is the minimal norm solution to equation \eqref{eq1}.
Recall that we are given the noisy data $f_\delta$, $\|f_\delta-f\|\le \delta$.

We assume in addition that
\begin{equation}
\label{eq44}
\| F^{(j)}(u)\| \le M_j(R,u_0),\qquad \forall u\in B(u_0,R),\quad 0\le j\le 2.
\end{equation}
This assumption is satisfied in many applications. 

\begin{lem}[\cite{R544} Lemma 2.7]
\label{lemma4}
Suppose $M_1, c_0$, and $c_1$ are positive constants and $0\not=y\in H$.
Then there exist $\lambda>0$ and a function $a(t)\in C^1[0,\infty)$, $0<a(t)\searrow 0$, such that 
the following conditions hold
\begin{align}
\label{eq45}
\frac{M_1}{\|y\|}&\le \lambda,\\
\label{eq46}
\frac{c_0}{a(t)}&\le \frac{\lambda}{2a(t)}\bigg{[}1-\frac{|\dot{a}(t)|}{a(t)}\bigg{]},\\
\label{eq47}
c_1\frac{|\dot{a}(t)|}{a(t)}&\le \frac{a(t)}{2\lambda}\bigg{[}1-\frac{|\dot{a}(t)|}{a(t)}\bigg{]},\\
\label{eq48}
\|F(0) - f_\delta\|& \le \frac{a^2(0)}{\lambda}.
\end{align}
\end{lem}

In the proof of Lemma 2.7 in \cite{R544} we 
have demonstrated that conditions \eqref{eq46}--\eqref{eq48}
are satisfied for $a(t)=\frac{d}{(c+t)^b}$, 
where $b\in(0,1]$,\, 
$c,d>0$ are constants, $c>6b$, 
and $d$ is sufficiently large.

\begin{thm}
\label{theorem5}
Assume 
$a(t)=\frac{d}{(c+t)^b}$, 
where $b\in(0,1]$,\, 
$c,d>0$ are constants, $c>6b$, 
and $d$ is sufficiently large so that conditions \eqref{eq46}--\eqref{eq48} hold.
Assume that $F:H\to H$ is a monotone operator, \eqref{eq44} holds, $u_0$ is an element of $H$, satisfying inequalities
\begin{equation}
\label{eqx83}
\|u_0-V_0\|\le \frac{\|F(0)-f_\delta\|}{a(0)},\quad h(0)=\|F(u_0) + a(0)u_0 -f_\delta\| \le \frac{1}{4} a(0)\|V_\delta(0)\|.
\end{equation}
Then the solution $u_\delta(t)$ to problem \eqref{eq43}
exists on an interval $[0,T_\delta]$,\, $\lim_{\delta\to0}T_\delta=\infty$, and 
there exists a unique $t_\delta$, $t_\delta\in (0,T_\delta)$ such that 
$\lim_{\delta\to 0}t_\delta=\infty$ and
\begin{equation}
\label{eq49}
\|F(u_\delta(t_\delta))-f_\delta\|=C_1\delta^\zeta,
\quad \|F(u_\delta(t)-f_\delta\|> C_1\delta^\zeta, \quad \forall t \in [0,t_\delta),
\end{equation}
where $C_1>1$ and $0<\zeta\le 1$. If $\zeta\in (0,1)$ and $t_\delta$ 
satisfies \eqref{eq49}, then
\begin{equation}
\label{eq50}
\lim_{\delta\to 0} \|u_\delta(t_\delta) - y\|=0.
\end{equation}
\end{thm}

\begin{rem}
{\rm 
One can choose $u_0$ satisfying  inequalities \eqref{eqx83} (see also \eqref{eq83}).
Indeed, if $u_0$ is a sufficiently close approximation to $V_\delta(0)$ the solution to 
equation \eqref{eq41.2} then inequalities \eqref{eqx83} are satisfied. 
Note that the second inequality in \eqref{eqx83} is a sufficient condition for 
\eqref{eq85}, i.e.,
\begin{equation}
\label{eqboa0}
e^{-\frac{t}{2}}h(0) 
\le \frac{1}{4} a(t)\|V_\delta(0)\|
\le \frac{1}{4} a(t)\|V_\delta(t)\|,\quad t\ge 0,
\end{equation}
to hold. In our proof inequality \eqref{eqboa0} (or inequality \eqref{eq85})
is used at $t=t_\delta$.
The stopping time $t_\delta$ is often sufficiently large for 
the quantity $e^{-\frac{t_\delta}{2}}h_0$ to be  small. In this case  
inequality \eqref{eqboa0} with $t=t_\delta$ is satisfied for a wide range of 
$u_0$.

Condition $c>6b$ is used in the proof of Lemma \ref{lemma21}.
}
\end{rem}

\subsubsection{The Dynamical system gradient method}

Denote
$$
A:=F'(u_\delta(t)),\quad A_a:=A + aI,\quad a=a(t),
$$
where $I$ is the identity operator, 
and $u_\delta(t)$ solves the following Cauchy problem:
\begin{equation}
\label{eq89}
\dot{u}_\delta = -A_{a(t)}^*[F(u_\delta)+a(t)u_\delta-f_\delta],\quad u_\delta(0)=u_0.
\end{equation}
Again, we assume in addition that
\begin{equation}
\label{eq44zxc}
\| F^{(j)}(u)\| \le M_j(R,u_0),\qquad \forall u\in B(u_0,R),\quad 0\le j\le 2.
\end{equation}
This assumption is satisfied in many applications.

Let us recall the following result: 
\begin{lem}[\cite{R549} Lemma 11]
\label{lemma6}
Suppose $M_1, c_0$, and $c_1$ are positive constants and $0\not=y\in H$.
Then there exist $\lambda>0$ and a function $a(t)\in C^1[0,\infty)$, $0<a(t)\searrow 0$, such that 
\begin{equation}
\label{eq90}
|\dot{a}(t)|\le \frac{a^3(t)}{4},
\end{equation}
and the following conditions hold
\begin{align}
\label{eq91}
\frac{M_1}{\|y\|}&\le \lambda,\\
\label{eq92}
c_0(M_1 + a(t))&\le \frac{\lambda}{2a^2(t)}\bigg{[}a^2(t)-\frac{2|\dot{a}(t)|}{a(t)}\bigg{]},\\
\label{eq93}
c_1\frac{|\dot{a}(t)|}{a(t)}&\le \frac{a^2(t)}{2\lambda}\bigg{[}a^2(t)-\frac{2|\dot{a}(t)|}{a(t)}\bigg{]},\\
\label{eq94}
\frac{\lambda}{a^2(0)}g(0)&<1.
\end{align}
\end{lem}

We have demonstrated in the proof of Lemma 11 in 
\cite{R549} that conditions
\eqref{eq90}--\eqref{eq94} are satisfied with $a(t)=\frac{d}{(c+t)^b}$, 
where $b\in(0,\frac{1}{4}]$,\, 
$c\geq 1$, and $d>0$ are constants, and $d$ is sufficiently large.

\begin{thm}
\label{theorem7}
Let $a(t)$ 
satisfy conditions \eqref{eq90}--\eqref{eq94} of Lemma~\ref{lemma6}. For example, one 
can choose $a(t)=\frac{d}{(c+t)^b}$, 
where $b\in(0,\frac{1}{4}]$,\, 
$c\geq 1$, and $d>0$ are constants, and $d$ is sufficiently large. 
Assume that $F:H\to H$ is a monotone operator, \eqref{eq44zxc} holds,
 $u_0$ is an element of $H$, satisfying inequalities 
\begin{equation}
\label{eqx83z}
h(0)=\|F(u_0) + a(0)u_0 -f_\delta\| \le \frac{1}{4} a(0)\|V_\delta(0)\|,
\end{equation}
Then the solution $u_\delta(t)$ to problem \eqref{eq89}
exists on an interval $[0,T_\delta]$,\, $\lim_{\delta\to0}T_\delta=\infty$, 
and 
there exists $t_\delta$, $t_\delta\in (0,T_\delta)$, not necessarily 
unique, such that 
\begin{equation}
\label{eq95}
\|F(u_\delta(t_\delta))-f_\delta\|=C_1\delta^\zeta,
\quad \lim_{\delta\to 0}t_\delta=\infty,
\end{equation}
where $C_1>1$ and $0<\zeta\le 1$ are constants. If $\zeta\in (0,1)$ and 
$t_\delta$ satisfies \eqref{eq95}, then
\begin{equation}
\label{eq96}
\lim_{\delta\to 0} \|u_\delta(t_\delta) - y\|=0.
\end{equation}
\end{thm}

\begin{rem}{\rm
One can easily choose $u_0$ satisfying inequality \eqref{eqx83z}.
Note that inequality \eqref{eqx83z} is a sufficient condition for 
inequality \eqref{eq134}, i.e., 
\begin{equation}
\label{eqboa}
e^{-\frac{t}{2}}h(0) 
\le \frac{1}{4} a(t)\|V_\delta(0)\|
\le \frac{1}{4} a(t)\|V_\delta(t)\|,\quad t\ge 0,
\end{equation}
to hold. 
In our proof inequality \eqref{eq134} (or \eqref{eqboa}) 
is used at $t=t_\delta$.
The stopping time $t_\delta$ is often sufficiently large for 
the quantity $e^{-\varphi(t_\delta)}h_0$ to be  small. In this case  
inequality \eqref{eqboa} (or \eqref{eq134}) with $t=t_\delta$ is satisfied for a wide range of 
$u_0$. The parameter $\zeta$ is not fixed in \eqref{eq95}. While we could 
fix it, for example, by setting $\zeta=0.9$, it is an interesting {\it 
open problem} to propose an optimal in some sense criterion for 
choosing $\zeta$.
}
\end{rem}

\subsubsection{The simple iteration DSM}

Let us consider a version of the DSM for solving equation \eqref{eq1}:
\begin{equation}
\label{eq138}
\dot{u}_\delta = -\big{(}F(u_\delta)+a(t)u_\delta - f_\delta\big{)},
\quad u_\delta(0)=u_0,
\end{equation}
where $F$ is a monotone operator. 

The advantage of this version compared with
\eqref{eq43} is the absence of the inverse operator in the algorithm, which
makes the algorithm \eqref{eq138} less expensive than \eqref{eq43}.
On the other hand, algorithm  \eqref{eq43} converges faster than  
\eqref{eq138} in many cases. The algorithm  \eqref{eq138} is cheaper than 
the DSM gradient algorithm proposed in \eqref{eq89}.  


The advantage of method \eqref{eq138}, a modified version of the simple 
iteration method, over the Gauss-Newton method and
the version \eqref{eq43} of the DSM is the following: neither inversion of 
matrices nor evaluation of
$F'$ is needed in a discretized version of \eqref{eq138}.
Although the convergence rate of the DSM \eqref{eq138} maybe slower than that of the DSM 
\eqref{eq43}, the DSM \eqref{eq138} might be faster than the DSM 
\eqref{eq43} for large-scale systems due to its lower computation 
cost.

In this Section we investigate a stopping rule based on a discrepancy 
principle (DP) for the DSM \eqref{eq138}. 
The main results of this Section is Theorem~\ref{theorem9} in which
a DP is formulated,  
the existence of a stopping time $t_\delta$ is proved, and
the convergence of the DSM with the proposed DP is justified under some 
natural assumptions.

Let us assume that 
\begin{equation}
\label{eq139}
\sup_{u\in B(u_0,R)} \|F'(u)\| \le M_1(u_0,R).
\end{equation}

\begin{lem}[\cite{R550} Lemma 11]
\label{lemma8}
Suppose $M_1$ and $c_1$ are positive constants and $0\not=y\in H$.
Then there exist a number $\lambda>0$ and a function $a(t)\in 
C^1[0,\infty)$, $0<a(t)\searrow 0$, such that 
\begin{equation}
\label{eq140}
|\dot{a}(t)|\le \frac{a^2(t)}{2},
\end{equation}
and the following conditions hold
\begin{align}
\label{eq141}
\frac{M_1}{\|y\|}&\le \lambda,\\
\label{eq142}
0&\le \frac{\lambda}{2a(t)}\bigg{[}a(t)-\frac{|\dot{a}(t)|}{a(t)}\bigg{]},\\
\label{eq143}
c_1\frac{|\dot{a}(t)|}{a(t)}&\le \frac{a(t)}{2\lambda}\bigg{[}a(t)-\frac{|\dot{a}(t)|}{a(t)}\bigg{]},\\
\label{eq144}
\frac{\lambda}{a(0)}g(0)&<1.
\end{align}
\end{lem}

It is shown in the proof of Lemma 11 in \cite{R550} that
conditions \eqref{eq140}--\eqref{eq144} hold for the 
function $a(t)=\frac{d}{(c+t)^b}$, 
where $b\in(0,\frac{1}{2}]$,\, 
$c\geq 1$ and $d>0$ are constants, and $d$ is sufficiently large.

\begin{thm}
\label{theorem9}
Let $a(t)$ 
satisfy conditions \eqref{eq140}--\eqref{eq144} in Lemma~\ref{lemma8}. For example, one 
can choose $a(t)=\frac{d}{(c+t)^b}$, 
where $b\in(0,\frac{1}{2}]$,\, 
$c\geq 1$ and $d>0$ are constants, and $d$ is sufficiently large. 
Assume that $F:H\to H$ is a monotone operator, condition \eqref{eq139} holds,
and $u_0$ is an element of $H$, 
satisfying inequality 
\begin{equation}
\label{eqx83a}
h(0)=\|F(u_0) + a(0)u_0 -f_\delta\| \le \frac{1}{4} a(0)\|V_\delta(0)\|,
\end{equation}
Assume that equation $F(u)=f$ has a solution $y\in B(u_0,R)$, possibly 
nonunique,
and $y$ is the minimal-norm solution to this equation.
Then the solution $u_\delta(t)$ to problem \eqref{eq138}
exists on an interval $[0,T_\delta]$,\, $\lim_{\delta\to0}T_\delta=\infty$, 
and 
there exists $t_\delta$, $t_\delta\in (0,T_\delta)$, not necessarily 
unique, such that 
\begin{equation}
\label{eq145}
\|F(u_\delta(t_\delta))-f_\delta\|=C_1\delta^\zeta,
\quad \lim_{\delta\to 0}t_\delta=\infty,
\end{equation}
where $C_1>1$ and $0<\zeta\le 1$ are constants. If $\zeta\in (0,1)$ and 
$t_\delta$ satisfies \eqref{eq145}, then
\begin{equation}
\label{eq146}
\lim_{\delta\to 0} \|u_\delta(t_\delta) - y\|=0.
\end{equation}
\end{thm}

\begin{rem}{\rm
One can easily choose $u_0$ satisfying  inequality \eqref{eqx83a} (see also \eqref{eq178}).
Again, inequality \eqref{eqx83a} is a sufficient condition for \eqref{eqboa}
(or \eqref{eq182}) to hold. In our proof inequality \eqref{eqboa}
is used at $t=t_\delta$.
The stopping time $t_\delta$ is often sufficiently large for 
the quantity $e^{-\varphi(t_\delta)}h_0$ to be  small. In this case  
inequality \eqref{eqboa} with $t=t_\delta$ is satisfied for a wide range of 
$u_0$. 
}
\end{rem}

\subsection{Iterative schemes} 

Let $0<a_n\searrow 0$ be a positive strictly decreasing sequence. 
Denote $V_n:=V_{n,\delta}$ where $V_{n,\delta}$ solves the following equation:
\begin{equation}
\label{eq186}
F(V_{n,\delta}) + a_n V_{n,\delta} - f_\delta = 0.
\end{equation}
Note that if $a_n:=a(t_n)$ then $V_{n,\delta}=V_\delta(t_n)$.

\subsubsection{Iterative scheme of Newton-type}

In this section we assume that $F$ is monotone operator, twice Fr\'{e}chet differentiable, and
\begin{equation}
\label{eq44zxcvb}
\| F^{(j)}(u)\| \le M_j(R,u_0),\qquad \forall u\in B(u_0,R),\quad 0\le j\le 2.
\end{equation}

Consider the following iterative scheme:
\begin{equation}
\label{eq187}
\begin{split}
u_{n+1} &= u_n - A_n^{-1}[F(u_n)+a_n u_n - f_\delta],\quad A_n:=F'(u_n)+ a_nI,\quad u_0=u_0,
\end{split}
\end{equation}
where $u_0$ is chosen so that inequality \eqref{eqchochet0} holds.
Note that $F'(u_n)\ge 0$ since $F$ is monotone. Thus, $\|A_n^{-1}\|\le\frac{1}{a_n}$.

\begin{lem}[\cite{R539} Lemma 2.5]
\label{lemma27}
Suppose $M_1, c_0$, and $c_1$ are positive constants and $0\not=y\in H$.
Then there exist $\lambda>0$ and a sequence $0<(a_n)_{n=0}^\infty\searrow 0$ such that the following conditions hold
\begin{align}
\label{eq343}
a_n &\le 2a_{n+1},\\
\label{eq344}
\|f_\delta -F(0)\| &\le \frac{a_0^2}{\lambda},\\
\label{eq345}
\frac{M_1}{\lambda}  &\le \|y\|,\\
\label{eq346}
\frac{a_n - a_{n+1}}{a_{n+1}^2} &\le \frac{1}{2c_1\lambda},\\
\label{eq347}
c_0\frac{a_n}{\lambda^2}+ \frac{a_n-a_{n+1}}{a_{n+1}}c_1 &\le \frac{a_{n+1}}{\lambda}.
\end{align}
\end{lem}

It is shown in the proof of Lemma 2.5 in \cite{R539} that 
conditions \eqref{eq343}--\eqref{eq347} hold for the sequence 
$a_n=\frac{d_0}{(d+n)^b}$, where $d\ge 1,\, 0<b\le 1$, and $d_0$ is 
sufficiently large.

\begin{rem}
\label{remark7}
{\rm In Lemmas~\ref{lemma27}--\ref{lemma29}, 
one can choose $a_0$ and $\lambda$ so that $\frac{a_0}{\lambda}$ is uniformly bounded
as $\delta\to 0$ even if $M_1(R)\to\infty$ as $R\to \infty$ at an arbitrary fast rate. 
Choices of $a_0$ and $\lambda$ to satisfy this condition are discussed in \cite{R539},
\cite{R549} and \cite{R550}. 
}
\end{rem}

Let $a_n$ and $\lambda$ 
satisfy conditions \eqref{eq343}--\eqref{eq347}.
Assume that equation $F(u)=f$ has a solution $y\in B(u_0,R)$, possibly nonunique,
and $y$ is the minimal-norm solution to this equation. 
Let $f$ be unknown but $f_\delta$ be given, and $\|f_\delta-f\|\le \delta$.
We have the following result:

\begin{thm}
\label{theorem10}
Assume $a_n=\frac{d_0}{(d+n)^b}$ where $d\ge 1,\, 0<b\le 1$, and $d_0$ is sufficiently large
so that conditions \eqref{eq343}--\eqref{eq347} hold. 
Let $u_n$ be defined by \eqref{eq187}. 
Assume that $u_0$ is chosen so that 
$\|F(u_0)-f_\delta\|>C_1\delta^\gamma>\delta$ and
\begin{equation}
\label{eqchochet0}
g_0:=\|u_0-V_0\| \le \frac{\|F(0)-f_\delta\|}{a_0}.
\end{equation}
Then there exists a unique $n_\delta$, depending on $C_1$ and $\gamma$ (see below), such that
\begin{equation}
\label{eq188}
\|F(u_{n_\delta})-f_\delta\|\le C_1\delta^\gamma,\quad
C_1\delta^\gamma < \|F(u_{n})-f_\delta\|,\quad \forall n< n_\delta,
\quad 
\end{equation}
where $C_1>1,\, 0<\gamma\le 1$.

Let $0<(\delta_m)_{m=1}^\infty$ be a sequence such that $\delta_m\to 0$. 
If $N$ is a cluster point of the sequence $n_{\delta_m}$ satisfying \eqref{eq188}, then
\begin{equation}
\label{eq189}
\lim_{m\to\infty} u_{n_{\delta_m}} = u^*,
\end{equation}
where $u^*$ is a solution to the equation $F(u)=f$.
If 
\begin{equation}
\label{eq190}
\lim_{m\to \infty}n_{\delta_m}=\infty,
\end{equation}
and  $\gamma\in (0,1)$, then
\begin{equation}
\label{eq191}
\lim_{m\to \infty} \|u_{n_{\delta_m}} - y\|=0.
\end{equation}
\end{thm}

Note that by Remark~\ref{remark8}, inequality \eqref{eqchochet0} is satisfied with $u_0=0$.

\subsubsection{An iterative scheme of gradient-type}

In this section we assume that $F$ is monotone operator, twice Fr\'{e}chet differentiable, and
\begin{equation}
\label{eq44zxcv}
\| F^{(j)}(u)\| \le M_j(R,u_0),\qquad \forall u\in B(u_0,R),\quad 0\le j\le 2.
\end{equation}

Consider the following iterative scheme:
\begin{equation}
\label{eq216}
\begin{split}
u_{n+1} &= u_n - \alpha_n A_n^*[F(u_n)+a_n u_n - f_\delta],\quad A_n:=F'(u_n)+ a_nI,\quad u_0=u_0,
\end{split}
\end{equation}
where $u_0$ is chosen so that inequality \eqref{eqchochet1} holds, and $\{\alpha_n\}_{n=1}^\infty$
is a positive sequence such that
\begin{equation}
\label{eq217}
0<\tilde{\alpha}\le \alpha_n \le \frac{2}{a_n^2 + (M_1+a_n)^2},
\qquad ||A_n||\leq M_1+a_n.
\end{equation}
It follows from this condition that
\begin{equation}
\label{eq218}
\|1-\alpha_n A_{a_n}^*A_{a_n}\|= \sup_{a_n^2\leq \lambda \leq 
(M_1+a_n)^2}|1-\alpha_n \lambda| 
\le 1 - \alpha_n a_n^2.
\end{equation} 
Note that  $F'(u_n)\ge 0$ since $F$ is monotone. 

\begin{lem}[\cite{R549} Lemma 12]
\label{lemma28}
Suppose $M_1, c_0, c_1$ and $\tilde{\alpha}$ are positive constants and $0\not=y\in H$.
Then there exist $\lambda>0$ and a sequence $0<(a_n)_{n=0}^\infty\searrow 0$ such that the following conditions hold
\begin{align}
\label{eq348}
\frac{a_n}{a_{n+1}} &\le 2,\\
\label{eq349}
\|f_\delta -F(0)\| &\le \frac{a_0^3}{\lambda},\\
\label{eq350}
\frac{M_1}{\lambda}  &\le \|y\|,\\
\label{eq351}
\frac{c_0(M_1+a_0)}{\lambda} &\le \frac{1}{2},\\
\label{eq352}
\frac{a_n^2}{\lambda}-\frac{\tilde{\alpha} a_n^4}{2\lambda} + \frac{a_n-a_{n+1}}{a_{n+1}}c_1 &\le \frac{a_{n+1}^2}{\lambda}.
\end{align}
\end{lem}

It is shown in the proof of Lemma~\ref{lemma28} in \cite{R549} that
the sequence satisfying conditions \eqref{eq348}--\eqref{eq352} 
can be chosen of the form $a_n=\frac{d}{(c+n)^b}$, where $c\ge 1,\, 0<b\le 
\frac{1}{4}$, and $d$ is sufficiently large.

Assume that equation $F(u)=f$ has a solution in $B(u_0,R)$, possibly nonunique,
and $y$ is the minimal-norm solution to this equation. 
Let $f$ be unknown but $f_\delta$ be given, and $\|f_\delta-f\|\le \delta$.
We prove the following result:

\begin{thm}
\label{theorem11}
Assume $a_n=\frac{d}{(c+n)^b}$ where $c\ge 1,\, 0<b\le \frac{1}{4}$, and $d$ is sufficiently large
so that conditions \eqref{eq348}--\eqref{eq352} hold. 
Let $u_n$ be defined by \eqref{eq216}. Assume that $u_0$ is chosen so that 
$\|F(u_0)-f_\delta\|>C_1\delta^\gamma>\delta$ and
\begin{equation}
\label{eqchochet1}
g_0:=\|u_0-V_0\| \le \frac{\|F(0)-f_\delta\|}{a_0}.
\end{equation}
Then there exists a unique $n_\delta$ such that
\begin{equation}
\label{eq219}
\|F(u_{n_\delta})-f_\delta\|\le C_1\delta^\zeta,\quad
C_1\delta^\zeta < \|F(u_{n})-f_\delta\|,\quad \forall n< n_\delta,
\quad 
\end{equation}
where $C_1>1,\, 0<\zeta\le 1$.

Let $0<(\delta_m)_{m=1}^\infty$ be a sequence such that $\delta_m\to 0$. 
If the sequence $\{n_m:=n_{\delta_m}\}_{m=1}^\infty$ is bounded, and $\{n_{m_j}\}_{j=1}^\infty$ 
is a convergent subsequence, then
\begin{equation}
\label{eq220}
\lim_{j\to\infty} u_{n_{m_j}} = \tilde{u},
\end{equation}
where $\tilde{u}$ is a solution to the equation $F(u)=f$.
If 
\begin{equation}
\label{eq221}
\lim_{m\to \infty}n_m=\infty,
\end{equation}
and  $\zeta\in (0,1)$, then
\begin{equation}
\label{eq222}
\lim_{m\to \infty} \|u_{n_m} - y\|=0.
\end{equation}
\end{thm}

It is pointed out in Remark~\ref{remark8} that inequality 
\eqref{eqchochet1} is satisfied with $u_0=0$.

\subsubsection{A simple iteration method}

In this section we assume that $F$ is monotone operator, Fr\'{e}chet differentiable.

Consider the following iterative scheme:
\begin{equation}
\label{eq247}
\begin{split}
u_{n+1} &= u_n - \alpha_n [F(u_n)+a_n u_n - f_\delta],\quad u_0=u_0,
\end{split}
\end{equation}
where $u_0$ is chosen so that inequality \eqref{eqchochet2} holds, and $\{\alpha_n\}_{n=1}^\infty$
is a positive sequence such that
\begin{equation}
\label{eq248}
0<\tilde{\alpha}\le \alpha_n \le \frac{2}{a_n + (M_1+a_n)},
\qquad M_1(u_0,R) = \sup_{u\in B(u_0,R)}\|F'(u)\|.
\end{equation}
It follows from this condition that
\begin{equation}
\label{eq249}
\|1-\alpha_n (J_n+a_n)\|= \sup_{a_n \leq \lambda \leq 
M_1+a_n}|1-\alpha_n \lambda| 
\le 1 - \alpha_n a_n.
\end{equation} 
Here, $J_n$ is an operator in $H$ such that $J_n\ge0$ and 
$\|J_n\|\le M_1,\, \forall u\in B(u_0,R)$. A specific choice
of  $J_n$ is made in formula  \eqref{eq256} below.

\begin{lem}
[\cite{R550} Lemma 12]
\label{lemma29}
Suppose $M_1$, $c_1$ and $\tilde{\alpha}$ are positive constants and 
$0\not=y\in H$.
Then there exist a number $\lambda>0$ and a sequence 
$0<(a_n)_{n=0}^\infty\searrow 0$ such that the following conditions hold
\begin{align}
\label{eq353}
\frac{a_n}{a_{n+1}} &\le 2,\\
\label{eq354}
\|f_\delta -F(0)\| &\le \frac{a_0^2}{\lambda},\\
\label{eq355}
\frac{M_1}{\lambda}  &\le \|y\|,\\
\label{eq356}
\frac{a_n}{\lambda}-\frac{\tilde{\alpha} a_n^2}{\lambda} + \frac{a_n-a_{n+1}}{a_{n+1}}c_1 &\le \frac{a_{n+1}}{\lambda}.
\end{align}
\end{lem}

It is shown in the proof of Lemma 12 in \cite{R550} that conditions 
\eqref{eq353}--\eqref{eq356} hold for
the sequence $a_n=\frac{d}{(c+n)^b}$, where $c\ge 1,\, 0<b\le 
\frac{1}{2}$, and $d$ is sufficiently large.

Let $a_n$ and $\lambda$ 
satisfy conditions \eqref{eq353}--\eqref{eq356}.
Assume that equation $F(u)=f$ has a solution $y\in B(u_0,R)$, possibly nonunique,
and $y$ is the minimal-norm solution to this equation. 
Let $f$ be unknown but $f_\delta$ be given, and $\|f_\delta-f\|\le \delta$.
We prove the following result:

\begin{thm}
\label{theorem12}
Assume $a_n=\frac{d}{(c+n)^b}$ where $c\ge 1,\, 0<b\le \frac{1}{2}$, and $d$ is sufficiently large
so that conditions \eqref{eq353}--\eqref{eq356} hold. 
Let $u_n$ be defined by \eqref{eq247}. 
Assume that $u_0$ is chosen so that 
$\|F(u_0)-f_\delta\|>C_1\delta^\gamma>\delta$ and
\begin{equation}
\label{eqchochet2}
g_0:=\|u_0-V_0\| \le \frac{\|F(0)-f_\delta\|}{a_0}.
\end{equation}
Then there exists a unique $n_\delta$ such that
\begin{equation}
\label{eq250}
\|F(u_{n_\delta})-f_\delta\|\le C_1\delta^\zeta,\quad
C_1\delta^\zeta < \|F(u_{n})-f_\delta\|,\quad \forall n< n_\delta,
\quad 
\end{equation}
where $C_1>1,\, 0<\zeta\le 1$.

Let $0<(\delta_m)_{m=1}^\infty$ be a sequence such that $\delta_m\to 0$. 
If the sequence $\{n_m:=n_{\delta_m}\}_{m=1}^\infty$ is bounded, and $\{n_{m_j}\}_{j=1}^\infty$ 
is a convergent subsequence, then
\begin{equation}
\label{eq251}
\lim_{j\to\infty} u_{n_{m_j}} = \tilde{u},
\end{equation}
where $\tilde{u}$ is a solution to the equation $F(u)=f$.
If 
\begin{equation}
\label{eq252}
\lim_{m\to \infty}n_m=\infty,
\end{equation}
and  $\zeta\in (0,1)$, then
\begin{equation}
\label{eq253}
\lim_{m\to \infty} \|u_{n_m} - y\|=0.
\end{equation}
\end{thm}

According to Remark~\ref{remark8}, inequality \eqref{eqchochet2} is 
satisfied with $u_0=0$.

\subsection{Nonlinear inequalities}

\subsubsection{A nonlinear differential inequality}

In \cite{R499} the following differential inequality 
\begin{equation}
\label{eq278x}
\dot{g}(t) \le -\gamma(t)g(t) + \alpha(t)g^2(t) + \beta(t),\qquad t\ge \tau_0,
\end{equation}
was studied and applied to various evolution problems. In \eqref{eq278x} 
$\alpha(t),\beta(t),\gamma(t)$ and $g(t)$ are continuous nonnegative 
functions on $[\tau_0,\infty)$ where $\tau_0$
is a fixed number. In \cite{R499}, an upper bound for $g(t)$ is 
obtained  under some conditions on 
$\alpha,\beta,\gamma$. In \cite{R558} the following generalization 
of \eqref{eq278x}:
\begin{equation}
\label{eq278}
\dot{g}(t) \le -\gamma(t)g(t) + \alpha(t)g^p(t) + \beta(t),\qquad t\ge 
\tau_0,\quad p>1,
\end{equation}
is studied. 

We have the following result:
\begin{thm}[\cite{R558} Theorem 1]
\label{theorem13}
Let $\alpha(t),\beta(t)$ and $\gamma(t)$ be continuous 
functions on $[\tau_0,\infty)$ and $\alpha(t)>0,\forall t\ge \tau_0$.
Suppose there exists a function $\mu(t)>0$, $\mu\in C^1[\tau_0,\infty)$,
such that
\begin{align}
\label{eq279}
\frac{\alpha(t)}{\mu^p(t)}+ \beta(t) &\le \frac{1}{\mu(t)}\bigg{[}\gamma -
\frac{\dot{\mu}(t)}{\mu(t)}\bigg{]}.
\end{align}
Let $g(t)\ge 0$ be a solution to inequality \eqref{eq278} such that
\begin{equation}
\label{eq280}
\mu(\tau_0)g(\tau_0)    < 1.
\end{equation}
Then $g(t)$ exists globally and the following estimate holds:
\begin{equation}
\label{eq281}
0\le g(t) < \frac{1}{\mu(t)},\qquad \forall t\ge \tau_0.
\end{equation}
Consequently, if $\lim_{t\to\infty} \mu(t)=\infty$, then
\begin{equation}
\label{eq282}
\lim_{t\to\infty} g(t)= 0.
\end{equation}
\end{thm}

When $p=2$ we have the following corollary:
\begin{cor}[\cite{R499} p. 97]
\label{corollary14}
If there exists a monotonically growing function $\mu(t)$,
$$
\mu\in C^1[\tau_0,\infty),\quad \mu>0, \quad \lim_{t\to\infty} 
\mu(t)=\infty,
$$
such that
\begin{align}
\label{eq299}
0\le \alpha(t)&\le \frac{\mu(t)}{2}\bigg{[}\gamma -\frac{\dot{\mu}(t)}{\mu(t)}\bigg{]},
\qquad \dot{u}:=\frac{du}{dt},\\
\label{eq300}
\beta(t)      &\le \frac{1}{2\mu(t)}\bigg{[}\gamma -\frac{\dot{\mu}(t)}{\mu(t)}\bigg{]},\\
\label{eq301}
\mu(\tau_0)g(\tau_0)    &< 1,
\end{align}
where $\alpha(t),\beta(t),\gamma(t)$ and $g(t)$ are continuous 
nonnegative functions on $[\tau_0,\infty)$,  $\tau_0\ge 0$,
 and $g(t)$ satisfies \eqref{eq278x}, then the following estimate holds:
\begin{equation}
\label{eq302}
0\le g(t) < \frac{1}{\mu(t)},\qquad \forall t\ge \tau_0.
\end{equation}
If inequalities \eqref{eq299}--\eqref{eq301} hold on an interval 
$[\tau_0,T)$, then
$g(t)$ exists on this interval and inequality \eqref{eq302} holds on 
$[\tau_0,T)$.
\end{cor}

\subsubsection{A discrete version of the nonlinear inequality}

\begin{thm}[\cite{R558} Theorem 4]
\label{548lemma16}
Let $\alpha_n,\gamma_n$ and $g_n$ be nonnegative sequences 
of numbers, and the following inequality holds:
\begin{equation}
\label{558eq1}
\begin{split}
\frac{g_{n+1}-g_n}{h_n}&\le -\gamma_n g_n+
\alpha_n g_n^p +\beta_n,\qquad h_n > 0,\quad 0< h_n\gamma_n < 1,
\end{split}
\end{equation}
or, equivalently, 
\begin{equation}
\qquad g_{n+1}\le g_n(1-h_n\gamma_n) +
\alpha_n h_n g_n^p+h_n\beta_n,\qquad h_n > 0,\quad 0< h_n\gamma_n < 1.
\end{equation}
If there is a monotonically growing sequence of positive numbers 
$(\mu_n)_{n=1}^\infty$, such that the following conditions hold:
\begin{align}
\label{558eq3}
\frac{\alpha_n}{\mu_n^p}+\beta_n&\le \frac{1}{\mu_n}\bigg{(}\gamma_n -
\frac{\mu_{n+1}-\mu_n}{\mu_n h_n}\bigg{)},\\
\label{558eq2}
g_0&\le\frac{1}{\mu_0},
\end{align}
then
\begin{equation}
\label{558eq5}
0\leq g_n\le\frac{1}{\mu_n} \qquad \forall n\ge 0.
\end{equation}
Therefore, if $\lim_{n\to\infty}\mu_n =\infty$, then $\lim_{n\to\infty} 
g_n = 0$.
\end{thm}

\section{Auxiliary results}

\subsection{Auxiliary results from the theory of monotone operators}

Recall the following result (see e.g. \cite[p.112]{R499}):
\begin{lem}
\label{lemma16}
Assume that equation \eqref{eq1} is solvable, $y$ is its minimal-norm solution, assumption
\eqref{eq2} holds, and $F$ is continuous. Then
\begin{equation}
\lim_{a\to 0} \|V_{a}-y\| = 0,
\end{equation}
where $V_{a}:=V_{0,a}$ solves equation \eqref{eq9} with $\delta=0$.
\end{lem}

\subsection{Auxiliary results for the regularized equation}

\begin{lem}[\cite{R554} Lemma 2]
\label{lemma17}
Assume $\|F(0)-f_\delta\|>0$.
Let $a>0$, and $F$ be monotone.
Denote 
$$
\psi(a) :=\|V_{\delta,a} \|,\qquad \phi(a):=a\psi(a)=\|F(V_{\delta,a}) - f_\delta\|,
$$ 
where $V_{\delta,a}$ solves \eqref{eq9}. 
Then
$\psi(a)$ is decreasing, and $\phi(a)$ is increasing (in the strict 
sense).
\end{lem} 

\begin{lem}[\cite{R554} Lemma 3]
\label{lem0}
If $F$ is monotone and continuous, then
$\|V_{\delta,a}\|=O(\frac{1}{a})$ as $a\to\infty$, and
\begin{equation}
\label{4eq2}
\lim_{a\to\infty}\|F(V_{\delta,a})-f_\delta\|=\|F(0)-f_\delta\|.
\end{equation}
\end{lem}

\begin{lem}[\cite{R554} Lemma 4]
\label{lemma19}
Let $C>0$ and $\gamma\in (0,1]$ be constants such that $C\delta^\gamma>\delta$.
Suppose that $\|F(0)-f_\delta\|> C\delta^\gamma$. 
Then, there exists a unique $a(\delta)>0$ such 
that $\|F(V_{\delta,a(\delta)})-f_\delta\|=C\delta^\gamma$.
\end{lem}

\begin{lem}
\label{lemma23}
If $F$ is monotone then
\begin{equation}
\label{eq330}
\max \bigg{(}\|F(u_\delta)-F(V_\delta)\|, a\|u-v\|\bigg{)}\le \|F(u)-F(v) +a(u-v)\|,\qquad 
 \forall u,v\in H.
\end{equation}
\end{lem}

\begin{proof}
Denote
\begin{equation}
w:=F(u)-F(v) +a(u-v),\qquad h:= \|w\|.
\end{equation}
Since $\langle F(u)-F(v),u-v\rangle \ge 0$, one 
obtains from two equations 
\begin{equation}
\langle w, u-v\rangle=
\langle F(u)-F(v) +a(u-v),u-v 
\rangle,
\end{equation}
and  
\begin{equation}
\langle w,F(u)-F(v)\rangle=\|F(u)-F(v)\|^2 
+a
\langle u-v, F(u)-F(v) \rangle,
\end{equation}
the 
following
two inequalities:
\begin{equation}
\label{544beq35}
a\|u - v\|^2 \le \langle v, u-v \rangle \le 
\|u - v\|h,
\end{equation}
and
\begin{equation}
\label{544beq36}
\|F(u)-F(v)\|^2\le \langle v, F(u)-F(v) \rangle
\le h\|F(u)-F(v)\|.
\end{equation}
Inequalities \eqref{544beq35} and \eqref{544beq36} imply: 
\begin{equation}
a\|u-v\|\le h,\quad \|F(u)-F(v)\|\le h.
\end{equation} 
Lemma~\ref{lemma23} is proved. 
\end{proof}

\begin{lem}
\label{lemma20}
Let $t_0$ satisfy
\begin{equation}
\label{eq311}
\frac{\delta}{a(t_0)}= \frac{1}{C-1}\|y\|,\qquad C>1.
\end{equation} 
Then,
\begin{equation}
\label{eq312}
\|F(u(t_0)) - f_\delta\| \le C\delta,
\end{equation}
and
\begin{equation}
\label{eq313}
\|\dot{V_\delta}\|\le \frac{|\dot{a}|}{a}\|y\|\bigg{(}1+\frac{1}{C-1}\bigg{)},\quad \forall t\le t_0.
\end{equation}
\end{lem}

\begin{proof}
This $t_0$ exists and is unique since $a(t)>0$ monotonically decays 
to 0 as $t\to\infty$.
Since $a(t)>0$  monotonically decays, one has: 
\begin{equation}
\label{544eqthieu}
\frac{\delta}{a(t)}\le\frac{1}{C-1}\|y\|,\qquad 0\leq t\leq t_0.
\end{equation}
By Lemma~\ref{lemma19}
there exists $t_1$ such that 
\begin{equation}
\label{5443eq18}
\|F(V_\delta(t_1))-f_\delta\|=C\delta,\quad F(V_\delta(t_1))+a(t_1)V_\delta(t_1)-f_\delta=0.
\end{equation}
{\it We claim that $t_1\in[0,t_0]$.}

 Indeed, from \eqref{5443eq18} and \eqref{eq338} one gets
$$
C\delta=a(t_1)\|V_\delta(t_1)\|\le a(t_1)\bigg{(}\|y\|+ \frac{\delta}{a(t_1)}\bigg{)}=a(t_1)\|y\|+\delta,\quad C>1,
$$
so
$$
\delta\le \frac{a(t_1)\|y\|}{C-1}.
$$
Thus,
$$
\frac{\delta}{a(t_1)}\le \frac{\|y\|}{C-1}=\frac{\delta}{a(t_0)}.
$$
Since $a(t)\searrow 0$, one has $t_1\le t_0$. 

Differentiating both sides of \eqref{eq41.2} with respect to $t$, one obtains
$$
A_{a(t)}\dot{V_\delta} = -\dot{a}V_\delta.
$$
This 
and the relations 
$$
A_a:=F'(u)+aI,\quad F'(u):=A\ge 0,
$$
imply
\begin{equation}
\label{544beq24}
\|\dot{V_\delta}\|\le |\dot{a}|\|A_{a(t)}^{-1} V_\delta\|\le \frac{|\dot{a}|}{a}\|V_\delta\|\le 
\frac{|\dot{a}|}{a}\bigg{(}\|y\|+\frac{\delta}{a}\bigg{)}\le \frac{|\dot{a}|}{a}\|y\|\bigg{(}1+\frac{1}{C-1}\bigg{)},\quad \forall t\le t_0.
\end{equation}
Lemma~\ref{lemma20} is proved. 
\end{proof}

\begin{lem}
\label{lemma22}
Let $n_0$ satisfy the inequality:
\begin{equation}
\label{eq318}
\frac{\delta}{a_{n_0+1}}> \frac{1}{C-1}\|y\|\ge \frac{\delta}{a_{n_0}},\qquad C>1.
\end{equation} 
Then,
\begin{equation}
\label{eq319}
\|F(u_{n_0+1}) - f_\delta\|\le C\delta,
\end{equation}
\begin{equation}
\label{eq320}
\|V_n\| \le \|y\|\bigg{(}1+\frac{2}{C-1}\bigg{)},
\end{equation}
and
\begin{equation}
\label{eq321}
\|V_n - V_{n+1}\|\le \frac{a_n-a_{n+1}}{a_{n+1}}\|y\|\bigg{(}1+\frac{2}{C-1}\bigg{)},
\qquad \forall n\le n_0+1.
\end{equation}
\end{lem}

\begin{proof}
One has $\frac{a_n}{a_{n+1}}\le 2,\, \forall\, n\ge 0$. 
This and inequality \eqref{eq318}  imply
\begin{equation}
\frac{2}{C-1}\|y\|\ge \frac{2\delta}{a_{n_0}} >\frac{\delta}{a_{n_0+1}}> \frac{1}{C-1}\|y\|\ge \frac{\delta}{a_{n_0}},\qquad C>1.
\end{equation}
Thus,
\begin{equation}
\label{eqxyz}
\frac{2}{C-1}\|y\|> \frac{\delta}{a_{n}},\quad \forall n \le n_0 + 1.
\end{equation}
The number $n_0$, satisfying \eqref{eq318}, exists and is unique since $a_n>0$ monotonically decays to 0 as $n\to\infty$.
By Lemma~\ref{lemma19},
there exists  a number $n_1$ such that 
\begin{equation}
\label{5463eq18}
\|F(V_{n_1+1})-f_\delta\|\le C\delta < \|F(V_{n_1})-f_\delta\|, 
\end{equation}
where $V_n$ solves the equation $F(V_{n})+a_{n}V_{n}-f_\delta=0$. 
{\it We claim that $n_1\in[0,n_0]$.} Indeed, 
one has $\|F(V_{n_1})-f_\delta\|=a_{n_1}\|V_{n_1}\|$, and $\|V_{n_1}\|\le \|y\|+\frac{\delta}{a_{n_1}}$ 
(cf. \eqref{eq338}), so
\begin{equation}
\label{546eeq19}
C\delta < a_{n_1}\|V_{n_1}\|\le a_{n_1}\bigg{(}\|y\|+ \frac{\delta}{a_{n_1}}\bigg{)}=a_{n_1}\|y\|+\delta,\quad C>1.
\end{equation}
Therefore,
\begin{equation}
\label{546eeq20}
\delta < \frac{a_{n_1}\|y\|}{C-1}.
\end{equation}
Thus, by \eqref{546eeq19}, 
\begin{equation}
\label{546yeq36}
\frac{\delta}{a_{n_1}} < \frac{\|y\|}{C-1} < \frac{\delta}{a_{n_0+1}}.
\end{equation}
Here the last inequality is a consequence of \eqref{546eeq19}.
Since $a_n$ decreases monotonically, inequality \eqref{546yeq36} implies $n_1\le n_0$. 
This and Lemma~\ref{lemma17} implies
\begin{equation}
\|F(u_{n_0+1})-f_\delta\| \le \|F(u_{n_1+1})-f_\delta\| \le C\delta.
\end{equation}

One has
\begin{equation}
\label{546eeq21}
\begin{split}
a_{n+1} \|V_n-V_{n+1}\|^2 &= \langle (a_{n+1} - a_{n})  V_{n} - F(V_n) + F(V_{n+1}), V_n - V_{n+1} \rangle \\
&\le \langle (a_{n+1} - a_{n})  V_{n}, V_n - V_{n+1} \rangle \\
&\le (a_{n}-a_{n+1}) \|V_{n}\| \|V_n - V_{n+1}\|.
\end{split}
\end{equation}
By \eqref{eq338}, $\|V_n\|\le \|y\|+\frac{\delta}{a_n}$, and, by \eqref{eqxyz}, 
$\frac{\delta}{a_n}\le \frac{2\|y\|}{C-1}$ for all $ n\le n_0+1$. This implies \eqref{eq320}. 
Therefore,
\begin{equation}
\|V_n-V_{n+1}\| \le \frac{a_n-a_{n+1}}{a_{n+1}}\|V_{n}\|\le \frac{a_n-a_{n+1}}{a_{n+1}}\|y\|
\bigg{(}1+\frac{2}{C-1}\bigg{)},\quad \forall n\le n_0+1.
\end{equation}
Lemma~\ref{lemma22} is proved.
\end{proof}

\begin{lem} 
\label{lemma24}
Let $V_a:=V_{\delta,a}|_{\delta=0}$, so $F(V_a)+aV-f=0$. 
Let $y$ be the minimal-norm solution to equation \eqref{eq1}. 
Then
\begin{equation}
\label{eq337}
\|V_{\delta,a}-V_a\|\le \frac{\delta}{a},\quad \|V_a\|\le \|y\|,\qquad a>0,
\end{equation}
and 
\begin{equation}
\label{eq338}
\|V_{\delta,a}\|\le \|V_a\|+\frac{\delta}{a}\le \|y\|+\frac{\delta}{a},\qquad a>0.
\end{equation}
\end{lem}

\begin{proof}
>From \eqref{eq3} one gets
$$
F(V_{\delta,a}) - F(V_a) + a (V_{\delta,a}-V_a)=f- f_\delta.
$$
Multiply this equality by $(V_{\delta,a}-V_a)$ and use \eqref{eq2} to obtain
\begin{align*}
\delta \|V_{\delta,a}-V_a\| &\ge \langle f-f_\delta, V_{\delta,a}-V_a \rangle\\
&= \langle F(V_{\delta,a}) - F(V_a) + a (V_{\delta,a}-V_a), V_{\delta,a}-V_a) \rangle\\
&\ge a \|V_{\delta,a}-V_a\|^2.
\end{align*}
This implies the first inequality in \eqref{eq337}. 

Let us derive a uniform with respect to $a$ bound on $\|V_a\|$. From the equation 
$$
F(V_a) + a V_a -F(y)=0,
$$
and the monotonicity of $F$ one gets
$$
0=\langle F(V_a) + a V_a -F(y), V_a - y \rangle \ge a \langle V_a, V_a-y\rangle.
$$
This implies the desired bound:
\begin{equation}
\label{eq14.1}
\|V_a\| \le \|y\|,\qquad \forall a>0.
\end{equation}
Similar arguments one can find in \cite[p. 113]{R499}. 

Inequalities \eqref{eq338} follow from \eqref{eq337} and \eqref{eq14.1} and the triangle inequality.

Lemma~\ref{lemma24} is proved.
\end{proof}

\begin{lem}[\cite{R544} Lemma 2.11]
\label{lemma21}
Let $a(t)=\frac{d}{(c+t)^b}$ where $d,c,b>0$,\, 
$c\ge 6b$. One has
\begin{equation}
\label{544auxieq2}
e^{-\frac{t}{2}}\int_0^t e^\frac{s}{2}|\dot{a}(s)|\|V_\delta(s)\|ds \le
 \frac{1}{2}a(t)\|V_\delta(t)\|,\qquad t\ge 0.
\end{equation} 
\end{lem}

\begin{lem}[\cite{R549} Lemma 9]
\label{lemma26}
Let $a(t)=\frac{d}{(c+t)^b}$ where $b\in(0,\frac{1}{4}],\, d^2c^{1-2b}\ge 6b$. 
Define $\varphi(t) = \frac{1}{2}\int_0^t \frac{a^2(s)}{2}ds$. 
Then, one has
\begin{equation}
e^{-\varphi(t)} \int_0^t e^{\varphi(t)} |\dot{a}(s)|\|V_\delta(s)\|ds\le \frac{1}{2}a(t)\|V_\delta(t)\|.
\end{equation}
\end{lem}

\begin{lem}[\cite{R550} Lemma 9]
\label{lemma25} 
Let $a(t)=\frac{d}{(c+t)^b}$ where $b\in(0,\frac{1}{2}],\, dc^{1-b}\ge 6b$. 
Define $\varphi(t) = \frac{1}{2}\int_0^t \frac{a(s)}{2}ds$. 
Then, one has
\begin{equation}
e^{-\varphi(t)} \int_0^t e^{\varphi(t)} |\dot{a}(s)|\|V_\delta(s)\|ds\le \frac{1}{2}a(t)\|V_\delta(t)\|.
\end{equation}
\end{lem}


\begin{rem}
\label{remark8}
{\rm 
In theorems ~\ref{theorem10}--\ref{theorem12} we choose $u_0\in H$ such that
\begin{equation}
\label{eq358}
g_0:=\|u_0 - V_0\| \le \frac{\|F(0)-f_\delta\|}{a_0}.
\end{equation}
It is easy to choose $u_0$ satisfying this condition. Indeed, if, for example, $u_0=0$,
then by Lemma~\ref{lemma17} one gets
\begin{equation}
g_0 = \|V_0\| = \frac{a_0 \|V_0\|}{a_0} \le \frac{\|F(0)-f_\delta\|}{a_0}.
\end{equation}
If \eqref{eq358} and either \eqref{eq344} or \eqref{eq354} hold then
\begin{equation}
\label{eq360}
g_0 \le \frac{a_0}{\lambda}.
\end{equation}
This inequality is used in the proof of Theorem~\ref{theorem10} and \ref{theorem12}.

If \eqref{eq358} and \eqref{eq349} hold, then
\begin{equation}
\label{eq361}
g_0 \le \frac{a^2_0}{\lambda}.
\end{equation}
This inequality is used in the proof of Theorem~\ref{theorem11}. 
}
\end{rem}

\section{Proofs of the basic results}
 
\subsection{Proofs of the Discrepancy Principles}

\subsubsection{Proof of Theorem~\ref{thm1}}

\begin{proof}[Proof of Theorem \ref{thm1}]
The existence and uniqueness of $a(\delta)$ follow from Lemma~\ref{lemma19}. 
Let us show that
\begin{equation}
\label{eq12}
\lim_{\delta\to0} a(\delta) = 0.
\end{equation}
The triangle inequality, the first inequality in \eqref{eq337}, equality \eqref{eq10} and equality \eqref{eq9} imply
\begin{equation}
\label{eq13}
\begin{split}
a(\delta)\|V_{a(\delta)}\| &\le a(\delta) \big{(}\|V_{\delta,a(\delta)} - V_{a(\delta)}\| + \|V_{\delta,a(\delta)}\|\big{)}\\
&\le \delta + a(\delta)\|V_{\delta,a(\delta)}\| = \delta + C \delta^\gamma.
\end{split}
\end{equation}
>From inequality \eqref{eq13}, one gets
\begin{equation}
\label{eq14}
\lim_{\delta\to 0}a(\delta)\|V_{a(\delta)}\| = 0.
\end{equation}
It follows from Lemma~\ref{lemma17} with $f_\delta= f$, i.e., $\delta=0$, that 
 the function $\phi_0(a):=a\|V_{a}\|$ is nonnegative and strictly increasing on $(0,\infty)$.
 This and relation \eqref{eq14} imply: 
 \begin{equation}
\label{eq15}
\lim_{\delta\to 0} a(\delta)= 0.
\end{equation}

>From \eqref{eq10} and \eqref{eq338}, one gets
\begin{equation}
\label{eq16}
C\delta^\gamma = a\|V_{\delta,a}\|\le a(\delta)\|y\| + \delta.
\end{equation}
Thus, 
one gets:
\begin{equation}
\label{eq17}
 C\delta^\gamma - \delta \le a(\delta)\|y\|.
\end{equation}
If $\gamma<1$ then $C-\delta^{1-\gamma}>0$ for sufficiently small $\delta$.  This implies: 
\begin{equation}
\label{eq18}
0\le \lim_{\delta\to 0}\frac{\delta}{a(\delta)}\le \lim_{\delta\to 0}\frac{\delta^{1-\gamma}\|y\|}{C-\delta^{1-\gamma}} = 0. 
\end{equation}
By the triangle inequality and the first inequality \eqref{eq337}, one has
\begin{equation}
\label{eq19}
\|V_{\delta, a(\delta)} - y\|\le \|V_{a(\delta)} - y\| + \|V_{a(\delta)}-V_{\delta,a(\delta)}\| 
\le \|V_{a(\delta)} - y\| + \frac{\delta}{a(\delta)}.
\end{equation}
Relation \eqref{eq11} follows from \eqref{eq18}, \eqref{eq19} and 
Lemma~\ref{lemma16}.
\end{proof}

\subsubsection{Proof of Theorem~\ref{theorem2}}

\begin{proof}[Proof of Theorem~\ref{theorem2}]
By Lemma~\ref{lemma23}
\begin{equation}
\label{eq28}
a\|u-v\| \le \|F(u)-F(v)+ au-av\|,\qquad \forall v,u \in H,\quad \forall a>0.
\end{equation}
Using inequality \eqref{eq28} with $v=v_\delta$ and $u=V_{\delta,\alpha(\delta)}$, 
equation \eqref{eq3} with $a=\alpha(\delta)$, and inequality \eqref{eq25},
one gets
\begin{equation}
\label{eq29}
\begin{split}
\alpha(\delta) \|v_\delta - V_{\delta,\alpha(\delta)}\| &\le \|F(v_\delta)-F(V_{\delta,\alpha(\delta)})+\alpha(\delta) v_\delta - \alpha(\delta) V_{\delta,\alpha(\delta)}\|\\
&=\|F(v_\delta)+ \alpha(\delta) v_\delta -f_\delta\| \le \theta\delta.
\end{split}
\end{equation}
Therefore,
\begin{equation}
\label{eq30}
\|v_\delta - V_{\delta,\alpha(\delta)}\| \le \frac{\theta \delta}{\alpha(\delta)}.
\end{equation}
Using \eqref{eq338} and \eqref{eq30}, one gets:
\begin{equation}
\label{eq31}
\alpha(\delta) \|v_\delta\| \le \alpha(\delta) \|V_{\delta,\alpha(\delta)}\| + 
\alpha(\delta)\|v_\delta - V_{\delta,\alpha(\delta)}\| \le \theta \delta + 
\alpha(\delta) \|y\| + \delta.
\end{equation}
>From the triangle inequality and inequalities \eqref{eq25} and  
\eqref{eq26} one obtains: 
\begin{equation}
\label{eq32}
\alpha(\delta)\|v_\delta\| \ge \|F(v_\delta) - f_\delta\| - \|F(v_\delta)+ \alpha(\delta) v_\delta -f_\delta\|
\ge C_1\delta^\gamma - \theta \delta.
\end{equation}
Inequalities \eqref{eq31} and \eqref{eq32} imply
\begin{equation}
\label{eq33}
C_1\delta^\gamma - \theta \delta \le \theta \delta + 
\alpha(\delta) \|y\| + \delta.
\end{equation}
This inequality and the fact that $C_1-\delta^{1-\gamma} - 2\theta\delta^{1-\gamma}>0$ for sufficiently small $\delta$ and $0<\gamma<1$ imply
\begin{equation}
\label{34}
\frac{\delta}{\alpha(\delta)} \le \frac{\delta^{1-\gamma}\|y\|}
{C_1-\delta^{1-\gamma} - 2\theta\delta^{1-\gamma}}, \qquad  0<\delta\ll 1.
\end{equation}
Thus, one obtains 
\begin{equation}
\label{eq35}
\lim_{\delta\to 0} \frac{\delta}{\alpha(\delta)} = 0.
\end{equation}
>From the triangle inequality and inequalities \eqref{eq25}, \eqref{eq26} and \eqref{eq30}, one gets
\begin{equation*}
\label{eq36}
\begin{split}
\alpha(\delta)\|V_{\delta,\alpha(\delta)}\|
&\le \|F(v_\delta) - f_\delta\| + \|F(v_\delta)+\alpha(\delta)v_\delta -f_\delta\| + 
\alpha(\delta)\|v_\delta - V_{\delta,\alpha(\delta)}\|\\
&\le C_2 \delta^\gamma + \theta \delta + \theta \delta.
\end{split}
\end{equation*}
This inequality implies
\begin{equation}
\label{eq37}
\lim_{\delta\to 0}\alpha(\delta)\|V_{\delta,\alpha(\delta)}\| = 0.
\end{equation}
The triangle inequality and inequality \eqref{eq337} imply
\begin{equation}
\label{eq38}
\begin{split}
\alpha\|V_{\alpha}\| &\le \alpha \big{(}\|V_{\delta,\alpha} - V_{\alpha}\| + \|V_{\delta,\alpha}\|\big{)}\\
&\le \delta + \alpha\|V_{\delta,\alpha}\|.
\end{split}
\end{equation}
>From formulas \eqref{eq38} and \eqref{eq37}, one gets
\begin{equation}
\label{eq39}
\lim_{\delta\to 0}\alpha(\delta)\|V_{\alpha(\delta)}\| = 0.
\end{equation}
It follows from Lemma~\ref{lemma17} with $f_\delta= f$, i.e., $\delta=0$, that 
 the function $\phi_0(a):=a\|V_{a}\|$ is nonnegative and strictly increasing on $(0,\infty)$.
 This and relation \eqref{eq39} imply 
 \begin{equation}
\label{eq40}
\lim_{\delta\to 0} \alpha(\delta)= 0.
\end{equation}
>From the triangle inequality and inequalities \eqref{eq30} 
and \eqref{eq337} one obtains
\begin{equation}
\label{eq41.1}
\begin{split}
\|v_\delta - y\| &\le \|v_\delta - V_{\delta,\alpha(\delta)}\| + 
\|V_{\delta,\alpha(\delta)} - V_{\alpha(\delta)}\| + \|V_{\alpha(\delta)} 
- y\|\\
&\le \frac{\theta \delta}{\alpha(\delta)} + \frac{\delta}{\alpha(\delta)} +  
\|V_{\alpha(\delta)} - y\|,
\end{split}
\end{equation}
where $V_{\alpha(\delta)}$ solves equation (3) with $a=\alpha(\delta)$
and $f_\delta=f$.

The conclusion \eqref{eq27} follows from inequalities 
\eqref{eq35}, \eqref{eq40}, \eqref{eq41.1} and Lemma~\ref{lemma16}.
Theorem~\ref{theorem2} is proved. 
\end{proof}


\subsection{Proofs of convergence of the Dynamical Systems Method}

\subsubsection{Proof of Theorem~\ref{theorem5}}

\begin{proof}[Proof of Theorem~\ref{theorem5}]
Denote 
\begin{equation}
\label{eq51}
C:=\frac{C_1+1}{2}.
\end{equation}
Let 
\begin{equation}
\label{eq52}
w:=u_\delta-V_\delta,\quad g(t):=\|w\|.
\end{equation}
One has
\begin{equation}
\label{eq53}
\dot{w}=-\dot{V}_\delta-A_{a(t)}^{-1}\big{[}F(u_\delta)-F(V_\delta)+a(t)w\big{]}.
\end{equation}
We use Taylor's formula and get:
\begin{equation}
\label{eq54}
F(u_\delta)-F(V_\delta)+aw=A_a w+ K, \quad \|K\| \le\frac{M_2}{2}\|w\|^2,
\end{equation}
where $K:=F(u_\delta)-F(V_\delta)-Aw$, and $M_2$ is the constant from the estimate \eqref{eq44} and $A_a:=A+aI$.
Multiplying \eqref{eq53} by $w$ and using \eqref{eq54} one gets
\begin{equation}
\label{eq55}
g\dot{g}\le -g^2+\frac{M_2}{2}\|A_{a(t)}^{-1}\|g^3+\|\dot{V}_\delta\|g.
\end{equation}

Let $t_0$ be defined as follows 
\begin{equation}
\label{eq56}
\frac{\delta}{a(t_0)}= \frac{1}{C-1}\|y\|,\qquad C>1.
\end{equation} 
This and Lemma~\ref{lemma20} imply that inequalities \eqref{eq312} and \eqref{eq313} hold. 
Since $g\ge 0$, inequalities \eqref{eq55} and \eqref{eq313} imply, 
for all $t\in [0,t_0]$, that
\begin{equation}
\label{eq57}
\dot{g}\le -g(t)+\frac{c_0}{a(t)}g^2+\frac{|\dot{a}|}{a(t)}c_1,
\quad c_0=\frac{M_2}{2},\quad c_1=\|y\|\bigg{(}1+\frac{1}{C-1}\bigg{)}.
\end{equation}

Inequality \eqref{eq57} is of the type \eqref{eq278} with
\begin{equation}
\label{eq58}
\gamma(t)=1,\quad \alpha(t)=\frac{c_0}{a(t)},\quad \beta(t)=c_1\frac{|\dot{a}|}{a(t)}.
\end{equation}
Let us check assumptions \eqref{eq299}--\eqref{eq301}. 
Take
\begin{equation}
\label{eq59}
\mu(t)=\frac{\lambda}{a(t)},
\end{equation}
where $\lambda=const>0$ and satisfies conditions \eqref{eq45}--\eqref{eq48} in Lemma~\ref{lemma4}.
It follows that inequalities \eqref{eq299}--\eqref{eq301} hold. 
Since $u_0$ satisfies the first inequality in \eqref{eqx83}, one gets $g(0)\le \frac{a(0)}{\lambda}$, by Remark~\ref{remark8}.
This, inequalities \eqref{eq299}--\eqref{eq301}, and Corollary \ref{corollary14} 
yield
\begin{equation}
\label{eq60}
g(t)<\frac{a(t)}{\lambda},\quad \forall t\le t_0, \qquad 
g(t):=\|u_\delta(t)-V_\delta(t)\|.
\end{equation}
Therefore,
\begin{equation}
\label{eq61}
\begin{split}
\|F(u_\delta(t))-f_\delta\|\le& \|F(u_\delta(t))-F(V_\delta(t))\|+\|F(V_\delta(t))-f_\delta\|\\
\le& M_1g(t)+\|F(V_\delta(t))-f_\delta\|\\
\le& \frac{M_1a(t)}{\lambda} + \|F(V_\delta(t))-f_\delta\|,\qquad \forall t\le t_0.
\end{split}
\end{equation}
>From \eqref{eq56} and Lemma~\ref{lemma20}, one gets
\begin{equation}
\label{eq62}
\|F(V_\delta(t_0))-f_\delta\|\le \|F(V_\delta(t_1))-f_\delta\|= C\delta.
\end{equation}
This, inequality \eqref{eq61}, the inequality $\frac{M_1}{\lambda}\le \|y\|$ (see \eqref{eq45}), the relation \eqref{eq311},
 and the definition $C_1=2C-1$ (see \eqref{eq51}), imply
\begin{equation}
\label{eq63}
\begin{split}
\|F(u_\delta(t_0))-f_\delta\| 
\le& \frac{M_1a(t_0)}{\lambda} + C\delta\\
\le& \frac{M_1\delta (C-1)}{\lambda\|y\|} + C\delta\le (2C-1)\delta=C_1\delta.
\end{split}
\end{equation}
Thus, if 
\begin{equation}
\label{eq64}
\|F(u_\delta(0))-f_\delta\|> C_1\delta^\gamma,\quad 0<\gamma\le 1,
\end{equation}
then, by the continuity of the function $t\to 
\|F(u_\delta(t))-f_\delta\|$ on $[0,\infty)$, 
there exists $t_\delta \in (0,t_0)$ such that
\begin{equation}
\label{eq65}
\|F(u_\delta(t_\delta))-f_\delta\|=C_1\delta^\gamma
\end{equation}
for any given $\gamma\in (0,1]$, and any fixed $C_1>1$.

{\it Let us prove \eqref{eq50}}. 

>From \eqref{eq61}  with $t=t_\delta$, 
and from \eqref{eq338}, one gets
\begin{equation}
\label{eq66}
\begin{split}
C_1\delta^\zeta &\le M_1 \frac{a(t_\delta)}{\lambda} + 
a(t_\delta)\|V_\delta(t_\delta)\|\\
&\le M_1 \frac{a(t_\delta)}{\lambda} + \|y\|a(t_\delta)+\delta.
\end{split}
\end{equation}
Thus, for sufficiently small $\delta$, one gets
\begin{equation}
\label{eq67}
\tilde{C}\delta^\zeta \le a(t_\delta) 
\bigg{(}\frac{M_1}{\lambda}+\|y\|\bigg{)},\quad \tilde{C}>0,
\end{equation}
where $\tilde{C}< C_1$ is a constant. 
Therefore, 
\begin{equation}
\label{eq68}
\lim_{\delta\to 0} \frac{\delta}{a(t_\delta)}\le
\lim_{\delta\to 0} 
\frac{\delta^{1-\zeta}}{\tilde{C}}\bigg{(}\frac{M_1}{\lambda}+\|y\|\bigg{)}
=0,\quad 0<\zeta<1.
\end{equation}

{\it We claim that}
\begin{equation}
\label{eq69}
\lim_{\delta\to0}t_\delta = \infty.
\end{equation}
Let us prove \eqref{eq69}. 
Using \eqref{eq43}, one obtains:
\begin{equation}
\label{eq70}
\frac{d}{dt}\big{(}F(u_\delta)+au_\delta - f_\delta\big{)}= A_a\dot{u}_\delta + \dot{a}u_\delta
= -\big{(}F(u_\delta)+au_\delta - f_\delta\big{)} + \dot{a}u_\delta.
\end{equation}
This and \eqref{eq41.2} imply:
\begin{equation}
\label{eq71}
\frac{d}{dt}\big{[}F(u_\delta)-F(V_\delta)+a(u_\delta-V_\delta)\big{]}
= -\big{[}F(u_\delta)-F(V_\delta) + a(u_\delta - V_\delta)\big{]} + \dot{a}u_\delta.
\end{equation}
Denote 
\begin{equation}
\label{eq72}
v:=v(t):=F(u_\delta(t))-F(V_\delta(t))+a(t)(u_\delta(t)-V_\delta(t)),\qquad 
h:=h(t):=\|v\|.
\end{equation}
Multiplying \eqref{eq71} by $v$, one obtains
\begin{equation}
\label{eq73}
\begin{split}
h\dot{h} &= -h^2 +\langle v,\dot{a}(u_\delta-V_\delta)\rangle + \dot{a}
\langle v,V_\delta\rangle\\ 
&\le -h^2 + h|\dot{a}|\|u_\delta-V_\delta\| + |\dot{a}|h\|V_\delta\|,
\qquad h\ge 0.
\end{split}
\end{equation}
Thus,
\begin{equation}
\label{eq74}
\dot{h}\le -h + |\dot{a}|\|u_\delta - V_\delta\| + |\dot{a}|\|V_\delta\|.
\end{equation}
Note that from inequality \eqref{eq81} one has
\begin{equation}
\label{eq75}
a\|u_\delta-V_\delta\|\le h,\quad \|F(u_\delta)-F(V_\delta)\|\le h.
\end{equation} 
Inequalities \eqref{eq74} and \eqref{eq75} imply
\begin{equation}
\label{eq76}
\dot{h} \le -h\bigg{(}1-\frac{|\dot{a}|}{a}\bigg{)} +|\dot{a}|\|V_\delta\|. 
\end{equation}
Since $1-\frac{|\dot{a}|}{a}\ge \frac{1}{2}$ because $c\ge 2b$, inequality \eqref{eq76} holds if
\begin{equation}
\label{eq77}
\dot{h} \le -\frac{1}{2}h + |\dot{a}|\|V_\delta\|.
\end{equation}
Inequality \eqref{eq77} implies:
\begin{equation}
\label{eq78}
h(t)\le h(0)e^{-\frac{t}{2}} + e^{-\frac{t}{2}}\int_0^t e^{\frac{s}{2}}|\dot{a}|\|V_\delta\|ds.
\end{equation}
>From \eqref{eq78} and \eqref{eq81}, one gets
\begin{equation}
\label{eq79}
\|F(u_\delta(t))-F(V_\delta(t))\| \le 
h(0)e^{-\frac{t}{2}} + e^{-\frac{t}{2}}\int_0^t e^{\frac{s}{2}}|\dot{a}|\|V_\delta\|ds.
\end{equation}
Therefore,
\begin{equation}
\label{eq80}
\begin{split}
\|F(u_\delta(t))-f_\delta\|&\ge \|F(V_\delta(t))-f_\delta\|-\|F(V_\delta(t))-F(u_\delta(t))\|\\
&\ge a(t)\|V_\delta(t)\| - h(0)e^{-\frac{t}{2}} - 
e^{-\frac{t}{2}}\int_0^t e^{\frac{s}{2}} |\dot{a}| \|V_\delta\|ds.
\end{split}
\end{equation}
>From Lemma~\ref{lemma21} it follows that 
there exists an 
$a(t)$ such that 
\begin{equation}
\label{eq81}
\frac{1}{2}a(t)\|V_\delta(t)\| \ge  e^{-\frac{t}{2}}\int_0^te^\frac{s}{2}|\dot{a}| \|V_\delta(s)\|ds.
\end{equation}
For example, one can choose 
\begin{equation}
\label{eq82}
a(t)=\frac{d}{(c+t)^b}, \quad 6b<c,
\end{equation}
where 
$d,c,b>0$.
Moreover, one can always choose $u_0$ such that 
\begin{equation}
\label{eq83}
h(0)=\|F(u_0) + a(0)u_0 -f_\delta\| \le \frac{1}{4} a(0)\|V_\delta(0)\|,
\end{equation}
because the equation $F(u_0) + a(0)u_0 -f_\delta=0$ is solvable.
If \eqref{eq83} holds, then
\begin{equation}
\label{eq84}
h(0)e^{-\frac{t}{2}}\le \frac{1}{4}a(0)\|V_\delta(0)\|e^{-\frac{t}{2}},\qquad 
t\ge 0.
\end{equation}
If $2b<c$, then  \eqref{eq82} implies $$e^{-\frac{t}{2}}a(0)\le a(t).$$
 Therefore,
\begin{equation}
\label{eq85}
e^{-\frac{t}{2}}h(0) 
\le \frac{1}{4} a(t)\|V_\delta(0)\|
\le \frac{1}{4} a(t)\|V_\delta(t)\|,\quad t\ge 0,
\end{equation}
where we have used the inequality $\|V_\delta(t)\|\le \|V_\delta(t')\|$ for 
$t<t'$, established in Lemma~\ref{lemma17}.
>From \eqref{eq65} and \eqref{eq69}--\eqref{eq85}, one gets
\begin{equation}
\label{eq86}
C_1\delta^\zeta = \|F(u_\delta(t_\delta))-f_\delta\|\ge 
\frac{1}{4}a(t_\delta)\|V_\delta(t_\delta)\|.
\end{equation}
Thus,
\begin{equation}
\label{eq87}
\lim_{\delta\to0}a(t_\delta)\|V_\delta(t_\delta)\|\le 
\lim_{\delta\to0}4C_1\delta^\zeta = 0.
\end{equation}
Since $\|V_\delta(t)\|$ increases (see Lemma~\ref{lemma17}), the above
formula implies 
$\lim_{\delta\to0}a(t_\delta)=0$. Since $0<a(t)\searrow 0$, 
it follows that $\lim_{\delta \to 0}t_\delta=\infty$, i.e.,  
\eqref{eq69} holds. 

It is now easy to finish the proof of the Theorem~\ref{theorem5}.

>From the triangle inequality and inequalities \eqref{eq60} and 
\eqref{eq337} one obtains
\begin{equation}
\label{eq88}
\begin{split}
\|u_\delta(t_\delta) - y\| &\le \|u_{\delta}(t_\delta) - V_{\delta}\| + 
\|V(t_\delta) - V_\delta(t_\delta)\| + \|V(t_\delta) - y\|\\
&\le \frac{a(t_\delta)}{\lambda} + \frac{\delta}{a(t_\delta)} + \|V(t_\delta)-y\|.
\end{split}
\end{equation}
Note that $V(t):=V_\delta(t)|_{\delta=0}$ and $V_\delta(t)$ solves \eqref{eq41.2}.
Note that $V(t_\delta) = V_{0,a(t_\delta)}$ (see equation \eqref{eq41.2}). 
>From \eqref{eq68}, \eqref{eq69}, inequality \eqref{eq88} and Lemma~\ref{lemma16}, one obtains
\eqref{eq50}.
Theorem~\ref{theorem5} is proved.
\end{proof}

\begin{rem}
\label{hehe}
{\rm
The trajectory $u_\delta(t)$ remains in the ball
$B(u_0,R):=\{u: \|u-u_0\|<R\}$ for all $t\leq t_\delta$, where $R$ does
not depend on $\delta$ as $\delta\to 0$. Indeed,
estimates \eqref{eq60}, \eqref{eq338} and \eqref{544eqthieu} imply:
\begin{equation}
\begin{split}
\|u_\delta(t)-u_0\|&\le
\|u_\delta(t)-V_\delta(t)\|+\|V_\delta(t)\|+\|u_0\|\\
&\le
\frac{a(0)}{\lambda} +\frac{C\|y\|}{C-1} +\|u_0\|:=R,
\qquad \forall t\le t_\delta.
\end{split}
\end{equation}
Here we have used the fact that $t_\delta < t_0$ (see Lemma~\ref{lemma20}).
Since one can choose $a(t)$ and $\lambda$ so that $\frac{a(0)}{\lambda}$
is uniformly bounded as $\delta\to 0$ and regardless of the growth of
$M_1$ (see Remark~\ref{remark7}) one concludes that $R$ can be chosen
independent of $\delta$ and $M_1$. 
}
\end{rem}


\subsubsection{Proof of Theorem~\ref{theorem7}}

\begin{proof}[Proof of Theorem~\ref{theorem7}]
 Denote 
\begin{equation}
\label{eq97}
C:=\frac{C_1+1}{2}.
\end{equation}
Let 
\begin{equation}
\label{eq98}
w:=u_\delta-V_\delta,\quad g(t):=\|w\|.
\end{equation}
One has
\begin{equation}
\label{eq99}
\dot{w}=-\dot{V}_\delta-A_{a(t)}^*\big{[}F(u_\delta)-F(V_\delta)+a(t)w\big{]}.
\end{equation}
We use Taylor's formula and get:
\begin{equation}
\label{eq100}
F(u_\delta)-F(V_\delta)+aw=A_a w+ K, \quad \|K\| \le\frac{M_2}{2}\|w\|^2,
\end{equation}
where $K:=F(u_\delta)-F(V_\delta)-Aw$, and $M_2$ is the constant from the estimate \eqref{eq44} and $A_a:=A+aI$.
Multiplying \eqref{eq99} by $w$ and using \eqref{eq100} one gets
\begin{equation}
\label{eq101}
g\dot{g}\le -a^2g^2+\frac{M_2(M_1+a)}{2}g^3+\|\dot{V}_\delta\|g,\quad
g:=g(t):=\|w(t)\|,
\end{equation}
where the estimates: $\langle A_a^* A_a w,w\rangle \ge a^2 
g^2$ and
$\|A_a\|\le M_1+a$ were used. Note that the inequality $\langle A_a^* A_a 
w,w\rangle \ge a^2 g^2$ is true if $A\ge 0$.
Since $F$ is monotone and differentiable (see \eqref{eq2}), 
one has $A:=F'(u_\delta)\ge 0$.

Let $t_0>0$ be such that 
\begin{equation}
\label{eq102}
\frac{\delta}{a(t_0)}= \frac{1}{C-1}\|y\|,\qquad C>1,
\end{equation} 
as in \eqref{eq311}.
It follows from Lemma~\ref{lemma20} that inequalities \eqref{eq312} and \eqref{eq313} hold. 

Since $g\ge 0$, inequalities \eqref{eq101} and \eqref{eq313} imply, 
for all $t\in[0,t_0]$, that
\begin{equation}
\label{eq103}
\dot{g}(t)\le 
-a^2(t)g(t)+c_0(M_1+a(t))g^2(t)+\frac{|\dot{a}(t)|}{a(t)}c_1,
\quad c_0=\frac{M_2}{2},\, c_1=\|y\|\bigg{(}1+\frac{1}{C-1}\bigg{)}.
\end{equation}

Inequality \eqref{eq103} is of the type \eqref{eq278} with
\begin{equation}
\label{eq104}
\gamma(t)=a^2(t),\quad \alpha(t)=c_0(M_1+a(t)),\quad \beta(t)=c_1\frac{|\dot{a}(t)|}{a(t)}.
\end{equation}
Let us check assumptions \eqref{eq299}--\eqref{eq301}. Take
\begin{equation}
\label{eq105}
\mu(t)=\frac{\lambda}{a^2(t)},\quad \lambda =\text{const}.
\end{equation}
By Lemma~\ref{lemma6} there exist $\lambda$ and $a(t)$ such that conditions \eqref{eq90}--\eqref{eq94}
hold. This implies that inequalities \eqref{eq299}--\eqref{eq301} hold. 
Thus, Corollary~\ref{corollary14} yields
\begin{equation}
\label{eq106}
g(t)<\frac{a^2(t)}{\lambda},\quad \forall t\le t_0.
\end{equation}
Note that inequality \eqref{eq106} holds for $t=0$ since \eqref{eq94} holds. 
Therefore,
\begin{equation}
\label{eq107}
\begin{split}
\|F(u_\delta(t))-f_\delta\|\le& \|F(u_\delta(t))-F(V_\delta(t))\|+\|F(V_\delta(t))-f_\delta\|\\
\le& M_1g(t)+\|F(V_\delta(t))-f_\delta\|\\
\le& \frac{M_1a^2(t)}{\lambda} + \|F(V_\delta(t))-f_\delta\|,\qquad \forall t\le t_0.
\end{split}
\end{equation}
It follows from Lemma~\ref{lemma17} that $\|F(V_\delta(t))-f_\delta\|$ is decreasing. 
Since $t_1\le t_0$, one gets 
\begin{equation}
\label{eq108}
\|F(V_\delta(t_0))-f_\delta\|\le \|F(V_\delta(t_1))-f_\delta\|= C\delta.
\end{equation}
This, inequality \eqref{eq107}, the inequality $\frac{M_1}{\lambda}\le \|y\|$ (see \eqref{eq91}), the relation \eqref{eq102},
 and the definition $C_1=2C-1$ (see \eqref{eq97}) imply
\begin{equation}
\label{eq109}
\begin{split}
\|F(u_\delta(t_0))-f_\delta\| 
\le& \frac{M_1a^2(t_0)}{\lambda} + C\delta\\
\le& \frac{M_1\delta (C-1)}{\lambda\|y\|} + C\delta\le (2C-1)\delta=C_1\delta.
\end{split}
\end{equation}
We have used the inequality 
\begin{equation}
\label{eq110}
a^2(t_0)\le a(t_0)=\frac{\delta (C-1)}{\|y\|}
\end{equation}
which
is true if $\delta$ is sufficiently small, or, equivalently, if $t_0$
is sufficiently large.
Thus, if 
\begin{equation}
\label{eq111}
\|F(u_\delta(0))-f_\delta\|\ge C_1\delta^\zeta,\quad 0<\zeta\le 1,
\end{equation}
then there exists $t_\delta \in (0,t_0)$ such that
\begin{equation}
\label{eq112}
\|F(u_\delta(t_\delta))-f_\delta\|=C_1\delta^\zeta
\end{equation}
for any given $\zeta\in (0,1]$, and any fixed $C_1>1$.

{\it Let us prove \eqref{eq96}. If this is done, then Theorem~\ref{theorem7} is 
proved.} 

First, we prove that $\lim_{\delta \to 0}\frac {\delta}{a(t_\delta)}=0.$

>From \eqref{eq107}  with $t=t_\delta$, \eqref{eq41.2} and \eqref{eq338}, one gets
\begin{equation}
\label{eq113}
\begin{split}
C_1\delta^\zeta &\le M_1 \frac{a^2(t_\delta)}{\lambda} + a(t_\delta)\|V_\delta(t_\delta)\|\\
&\le M_1 \frac{a^2(t_\delta)}{\lambda} + \|y\|a(t_\delta)+\delta.
\end{split}
\end{equation}
Thus, for sufficiently small $\delta$, one gets
\begin{equation}
\label{eq114}
\tilde{C}\delta^\zeta \le a(t_\delta) \bigg{(}\frac{M_1a(0)}{\lambda}+\|y\|\bigg{)},\quad \tilde{C}>0,
\end{equation}
where $\tilde{C}<C_1$ is a constant. 
Therefore, 
\begin{equation}
\label{eq115}
\lim_{\delta\to 0} \frac{\delta}{a(t_\delta)}\le
\lim_{\delta\to 0} \frac{\delta^{1-\zeta}}{\tilde{C}}\bigg{(}\frac{M_1a(0)}{\lambda}+\|y\|\bigg{)}
=0,\quad 0<\zeta<1.
\end{equation}

{\it Secondly, we prove that}
\begin{equation}
\label{eq116}
\lim_{\delta\to0}t_\delta = \infty.
\end{equation}
Using \eqref{eq89}, one obtains:
\begin{equation}
\label{eq117}
\frac{d}{dt}\big{(}F(u_\delta)+au_\delta - f_\delta\big{)}= A_a\dot{u}_\delta + \dot{a}u_\delta
= -A_a A_a^*\big{(}F(u_\delta)+au_\delta - f_\delta\big{)} + \dot{a}u_\delta.
\end{equation}
This and \eqref{eq41.2} imply:
\begin{equation}
\label{eq118}
\frac{d}{dt}\big{[}F(u_\delta)-F(V_\delta)+a(u_\delta-V_\delta)\big{]}
= - A_a A_a^*\big{[}F(u_\delta)-F(V_\delta) + a(u_\delta - V_\delta)\big{]} + \dot{a}u_\delta.
\end{equation}
Denote 
\begin{equation}
\label{eq119}
v:=F(u_\delta)-F(V_\delta)+a(u_\delta-V_\delta),\quad h:=h(t):=\|v(t)\|.
\end{equation}
Multiplying \eqref{eq118} by $v$ and using monotonicity of $F$, one obtains
\begin{equation}
\label{eq120}
\begin{split}
h\dot{h} &= -\langle A_a A_a^* v, v\rangle +\langle v,\dot{a}(u_\delta-V_\delta)\rangle + \dot{a}\langle v,V_\delta\rangle\\ 
&\le -h^2a^2 + h|\dot{a}|\|u_\delta-V_\delta\| + |\dot{a}|h\|V_\delta\|,\qquad h\ge 0.
\end{split}
\end{equation}
Again, we have used the inequality $A_aA_a^* \ge a^2$, which holds for 
$A\geq 0$, i.e., monotone operators $F$.  
Thus,
\begin{equation}
\label{eq121}
\dot{h}\le -ha^2 + |\dot{a}|\|u_\delta - V_\delta\| + |\dot{a}|\|V_\delta\|.
\end{equation}
>From inequality \eqref{eq330} we have 
\begin{equation}
\label{eq122}
a\|u_\delta-V_\delta\|\le h,\quad \|F(u_\delta)-F(V_\delta)\|\le h.
\end{equation} 
Inequalities \eqref{eq121} and \eqref{eq122} imply
\begin{equation}
\label{eq123}
\dot{h} \le -h\bigg{(}a^2-\frac{|\dot{a}|}{a}\bigg{)} +|\dot{a}|\|V_\delta\|. 
\end{equation}
Since $a^2-\frac{|\dot{a}|}{a}\ge \frac{3a^2}{4}>\frac{a^2}{2}$ 
by inequality \eqref{eq90}, 
it follows from inequality \eqref{eq123} that
\begin{equation}
\label{eq124}
\dot{h} \le -\frac{a^2}{2}h + |\dot{a}|\|V_\delta\|.
\end{equation}
Inequality \eqref{eq124} implies:
\begin{equation}
\label{eq125}
h(t)\le h(0)e^{-\int_0^t\frac{a^2(s)}{2}ds} + e^{-\int_0^t\frac{a^2(s)}{2}ds}
\int_0^t e^{\int_0^s\frac{a^2(\xi)}{2}d\xi}|\dot{a}(s)|\|V_\delta(s)\|ds.
\end{equation}
Denote 
$$\varphi(t):=\int_0^t\frac{a^2(s)}{2}ds.$$ 
>From \eqref{eq125} and 
\eqref{eq122}, one gets
\begin{equation}
\label{eq126}
\|F(u_\delta(t))-F(V_\delta(t))\| \le 
h(0)e^{-\varphi(t)} + e^{-\varphi(t)}\int_0^t 
e^{\varphi(s)}|\dot{a}(s)|\|V_\delta(s)\|ds.
\end{equation}
Therefore,
\begin{equation}
\label{eq127}
\begin{split}
\|F(u_\delta(t))-f_\delta\|&\ge \|F(V_\delta(t))-f_\delta\|-\|F(V_\delta(t))-F(u_\delta(t))\|\\
&\ge a(t)\|V_\delta(t)\| - h(0)e^{-\varphi(t)} - 
e^{-\varphi(t)}\int_0^t e^{\varphi(s)} |\dot{a}| \|V_\delta\|ds.
\end{split}
\end{equation}
>From Lemma~\ref{lemma26} it follows that there exists an 
$a(t)$ such that 
\begin{equation}
\label{eq128}
\frac{1}{2}a(t)\|V_\delta(t)\| \ge  e^{-\varphi(t)}\int_0^te^{\varphi(s)}|\dot{a}| \|V_\delta(s)\|ds.
\end{equation}
For example, one can choose 
\begin{equation}
\label{eq129}
a(t)=\frac{c_1}{(c+t)^b}, \quad b\in (0,\frac{1}{4}],\quad c_1^2 c^{1-2b}\ge 6b,
\end{equation}
where 
$c_1,c>0$.
Moreover, one can always choose $u_0$ such that 
\begin{equation}
\label{eq130}
h(0)=\|F(u_0) + a(0)u_0 -f_\delta\| \le \frac{1}{4} a(0)\|V_\delta(0)\|,
\end{equation}
because the equation 
\begin{equation}
\label{eq131}
F(u_0) + a(0)u_0 -f_\delta=0
\end{equation}
is solvable.

If \eqref{eq130} holds, then
\begin{equation}
\label{eq132}
h(0)e^{-\varphi(t)}\le \frac{1}{4}a(0)\|V_\delta(0)\|e^{-\varphi(t)},\qquad 
t\ge 0.
\end{equation}
If \eqref{eq129} holds, $c\ge1$  and $2b\le c_1^2$, then it follows that 
\begin{equation}
\label{eq133}
e^{-\varphi(t)}a(0)\le a(t).
\end{equation}
Indeed, inequality $a(0)\le a(t)e^{\varphi(t)}$ is obviously true for 
$t=0$, and $\big(a(t)e^{\varphi(t)}\big)'_t\geq 0$, provided that $c\geq 
1$ and $2b\le c_1^2$. 

Inequalities \eqref{eq132} and \eqref{eq133} imply
\begin{equation}
\label{eq134}
e^{-\varphi(t)}h(0) 
\le \frac{1}{4} a(t)\|V_\delta(0)\|
\le \frac{1}{4} a(t)\|V_\delta(t)\|,\quad t\ge 0.
\end{equation}
where we have used the inequality $\|V_\delta(t)\|\le \|V_\delta(t')\|$ for $t\le t'$, 
established in Lemma~\ref{lemma17}.
>From \eqref{eq112} and \eqref{eq116}--\eqref{eq134}, one gets
\begin{equation}
\label{eq135}
C\delta^\zeta = \|F(u_\delta(t_\delta))-f_\delta\|\ge \frac{1}{4}a(t_\delta)\|V_\delta(t_\delta)\|.
\end{equation}
Thus,
\begin{equation}
\label{eq136}
\lim_{\delta\to0}a(t_\delta)\|V_\delta(t_\delta)\|\le 
\lim_{\delta\to0}4C\delta^\zeta = 0.
\end{equation}
Since $\|V_\delta(t)\|$ is increasing, this implies 
$\lim_{\delta\to0}a(t_\delta)=0$. 
Since $0<a(t)\searrow 0$, it follows that  
\eqref{eq116} holds. 

>From the triangle inequality and inequalities \eqref{eq106} and \eqref{eq337} one obtains
\begin{equation}
\label{eq137}
\begin{split}
\|u_\delta(t_\delta) - y\| &\le \|u_{\delta}(t_\delta) - V_{\delta}\| + 
\|V(t_\delta) - V_\delta(t_\delta)\| + \|V(t_\delta) - y\|\\
&\le \frac{a^2(t_\delta)}{\lambda} + \frac{\delta}{a(t_\delta)} + \|V(t_\delta)-y\|,
\end{split}
\end{equation}
where $V(t):=V_\delta(t)|_{\delta=0}$ and $V_\delta(t)$ solves \eqref{eq41.2}.
>From \eqref{eq115}, \eqref{eq116}, inequality \eqref{eq137} and Lemma~\ref{lemma16}, one obtains
\eqref{eq96}. Theorem~\ref{theorem7} is proved.
\end{proof}

By the arguments, similar to the ones in the proof of 
Theorem~\ref{theorem10}--\ref{theorem12}
or in Remark~\ref{hehe}, one can show that
the trajectory $u_\delta(t)$ remains in the ball
$B(u_0,R):=\{u: \|u-u_0\|<R\}$ for all $t\leq t_\delta$, where $R$ does
not depend on $\delta$ as $\delta\to 0$.


\subsubsection{Proof of Theorem~\ref{theorem9}}

\begin{proof}[Proof of Theorem~\ref{theorem9}]
 Denote 
\begin{equation}
\label{eq147}
C:=\frac{C_1+1}{2}.
\end{equation}
Let 
\begin{equation}
\label{eq148}
w:=u_\delta-V_\delta,\qquad g:=g(t):=\|w(t)\|.
\end{equation}
One has
\begin{equation}
\label{eq149}
\dot{w}=-\dot{V}_\delta-\big{[}F(u_\delta)-F(V_\delta)+a(t)w\big{]}.
\end{equation}
Multiplying \eqref{eq149} by $w$ and using \eqref{eq2} one gets
\begin{equation}
\label{eq150}
g\dot{g}\le -ag^2+\|\dot{V}_\delta\|g.
\end{equation}

Let $t_0>0$ be such that 
\begin{equation}
\label{eq151}
\frac{\delta}{a(t_0)}= \frac{1}{C-1}\|y\|,\qquad C>1.
\end{equation} 
This $t_0$ exists and is unique since $a(t)>0$ monotonically decays to 0 as $t\to\infty$.
It follows from inequality \eqref{eq151} and Lemma~\ref{lemma20} that inequalities \eqref{eq312} and \eqref{eq313} hold. 

Since $g\ge 0$, inequalities \eqref{eq150} and \eqref{eq313} imply
\begin{equation}
\label{eq152}
\dot{g}\le -a(t)g(t) + \frac{|\dot{a}(t)|}{a(t)}c_1,
\quad c_1=\|y\|\bigg{(}1+\frac{1}{C-1}\bigg{)}.
\end{equation}

Inequality \eqref{eq152} is of the type \eqref{eq278} with
\begin{equation}
\label{eq153}
\gamma(t)=a(t),\quad \alpha(t)=0,\quad \beta(t)=c_1\frac{|\dot{a}(t)|}{a(t)}.
\end{equation}
Let us check assumptions \eqref{eq299}--\eqref{eq301}. Take
\begin{equation}
\label{eq154}
\mu(t)=\frac{\lambda}{a(t)},\quad \lambda =\text{const}.
\end{equation}
By Lemma~\ref{lemma8} there exist $\lambda$ and $a(t)$ such that conditions \eqref{eq141}--\eqref{eq144}
hold. It follows that inequalities \eqref{eq299}--\eqref{eq301} hold. 
Thus, Corollary~\ref{corollary14} yields
\begin{equation}
\label{eq155}
g(t)<\frac{a(t)}{\lambda},\quad \forall t\le t_0.
\end{equation}
Therefore,
\begin{equation}
\label{eq156}
\begin{split}
\|F(u_\delta(t))-f_\delta\|\le& \|F(u_\delta(t))-F(V_\delta(t))\|+\|F(V_\delta(t))-f_\delta\|\\
\le& M_1g(t)+\|F(V_\delta(t))-f_\delta\|\\
\le& \frac{M_1a(t)}{\lambda} + \|F(V_\delta(t))-f_\delta\|,\qquad \forall t\le t_0.
\end{split}
\end{equation}
It follows from Lemma~\ref{lemma17} that $\|F(V_\delta(t))-f_\delta\|$ is decreasing. 
Since $t_1\le t_0$, one gets 
\begin{equation}
\label{eq157}
\|F(V_\delta(t_0))-f_\delta\|\le \|F(V_\delta(t_1))-f_\delta\|= C\delta.
\end{equation}
This, inequality \eqref{eq156}, the inequality $\frac{M_1}{\lambda}\le \|y\|$ (see \eqref{eq141}), the relation \eqref{eq151},
 and the definition $C_1=2C-1$ (see \eqref{eq147}) imply
\begin{equation}
\label{eq158}
\begin{split}
\|F(u_\delta(t_0))-f_\delta\| 
\le& \frac{M_1a(t_0)}{\lambda} + C\delta\\
\le& \frac{M_1\delta (C-1)}{\lambda\|y\|} + C\delta\le (2C-1)\delta=C_1\delta.
\end{split}
\end{equation}
Thus, if 
\begin{equation}
\label{eq159}
\|F(u_\delta(0))-f_\delta\|\ge C_1\delta^\zeta,\quad 0<\zeta\le 1,
\end{equation}
then there exists $t_\delta \in (0,t_0)$ such that
\begin{equation}
\label{eq160}
\|F(u_\delta(t_\delta))-f_\delta\|=C_1\delta^\zeta
\end{equation}
for any given $\zeta\in (0,1]$, and any fixed $C_1>1$.

{\it Let us prove \eqref{eq146}. If this is done, then Theorem~\ref{theorem9} is 
proved.} 

First, we prove that $\lim_{\delta \to 0}\frac {\delta}{a(t_\delta)}=0.$

>From \eqref{eq156}  with $t=t_\delta$, and from \eqref{eq338}, one gets
\begin{equation}
\label{eq161}
\begin{split}
C_1\delta^\zeta &\le M_1 \frac{a(t_\delta)}{\lambda} + a(t_\delta)\|V_\delta(t_\delta)\|\\
&\le M_1 \frac{a(t_\delta)}{\lambda} + \|y\|a(t_\delta)+\delta.
\end{split}
\end{equation}
Thus, for sufficiently small $\delta$, one gets
\begin{equation}
\label{eq162}
\tilde{C}\delta^\zeta \le a(t_\delta) \bigg{(}\frac{M_1}{\lambda}+\|y\|\bigg{)},\quad \tilde{C}>0,
\end{equation}
where $\tilde{C}<C_1$ is a constant. 
Therefore, 
\begin{equation}
\label{eq163}
\lim_{\delta\to 0} \frac{\delta}{a(t_\delta)}\le
\lim_{\delta\to 0} \frac{\delta^{1-\zeta}}{\tilde{C}}\bigg{(}\frac{M_1}{\lambda}+\|y\|\bigg{)}
=0,\quad 0<\zeta<1.
\end{equation}

{\it Secondly, we prove that}
\begin{equation}
\label{eq164}
\lim_{\delta\to0}t_\delta = \infty.
\end{equation}
Using \eqref{eq138}, one obtains:
\begin{equation}
\label{eq165}
\frac{d}{dt}\big{(}F(u_\delta)+au_\delta - f_\delta\big{)}= A_a\dot{u}_\delta + \dot{a}u_\delta
= -A_a \big{(}F(u_\delta)+au_\delta - f_\delta\big{)} + \dot{a}u_\delta,
\end{equation}
where $A_a:=F'(u_\delta)+a$. 
This and \eqref{eq41.2} imply:
\begin{equation}
\label{eq166}
\frac{d}{dt}\big{[}F(u_\delta)-F(V_\delta)+a(u_\delta-V_\delta)\big{]}
= - A_a \big{[}F(u_\delta)-F(V_\delta) + a(u_\delta - V_\delta)\big{]} + \dot{a}u_\delta.
\end{equation}
Denote 
\begin{equation}
\label{eq167}
v:=F(u_\delta)-F(V_\delta)+a(u_\delta-V_\delta),\quad h=\|v\|.
\end{equation}
Multiplying \eqref{eq166} by $v$ and using monotonicity of $F$, one obtains
\begin{equation}
\label{eq168}
\begin{split}
h\dot{h} &= -\langle A_a v, v\rangle +\langle v,\dot{a}(u_\delta-V_\delta)
\rangle + \dot{a}\langle v,V_\delta\rangle\\ 
&\le -h^2a + h|\dot{a}|\|u_\delta-V_\delta\| + |\dot{a}|h\|V_\delta\|,
\qquad h\ge 0.
\end{split}
\end{equation}
Again, we have used the inequality $\langle F'(u_\delta)v,v\rangle \ge 0$ 
which follows from the monotonicity of $F$.  
Thus,
\begin{equation}
\label{eq169}
\dot{h}\le -ha + |\dot{a}|\|u_\delta - V_\delta\| + |\dot{a}|\|V_\delta\|.
\end{equation}

Inequalities \eqref{eq169} and \eqref{eq176} imply
\begin{equation}
\label{eq170}
\dot{h} \le -h\bigg{(}a-\frac{|\dot{a}|}{a}\bigg{)} +|\dot{a}|\|V_\delta\|. 
\end{equation}
Since $a-\frac{|\dot{a}|}{a}\ge \frac{a}{2}$ 
by inequality \eqref{eq140}, 
it follows from inequality \eqref{eq170} that
\begin{equation}
\label{eq171}
\dot{h} \le -\frac{a}{2}h + |\dot{a}|\|V_\delta\|.
\end{equation}
Inequality \eqref{eq171} implies:
\begin{equation}
\label{eq172}
h(t)\le h(0)e^{-\int_0^t\frac{a(s)}{2}ds} + e^{-\int_0^t\frac{a(s)}{2}ds}
\int_0^t e^{\int_0^s\frac{a(\xi)}{2}d\xi}|\dot{a}(s)|\|V_\delta(s)\|ds.
\end{equation}
Denote 
\begin{equation}
\label{eq173}
\varphi(t):=\int_0^t\frac{a(s)}{2}ds.
\end{equation}
>From \eqref{eq172} and 
\eqref{eq176}, one gets
\begin{equation}
\label{eq174}
\|F(u_\delta(t))-F(V_\delta(t))\| \le 
h(0)e^{-\varphi(t)} + e^{-\varphi(t)}\int_0^t 
e^{\varphi(s)}|\dot{a}(s)|\|V_\delta(s)\|ds.
\end{equation}
Therefore,
\begin{equation}
\label{eq175}
\begin{split}
\|F(u_\delta(t))-f_\delta\|&\ge \|F(V_\delta(t))-f_\delta\|-\|F(V_\delta(t))-F(u_\delta(t))\|\\
&\ge a(t)\|V_\delta(t)\| - h(0)e^{-\varphi(t)} - 
e^{-\varphi(t)}\int_0^t e^{\varphi(s)} |\dot{a}| \|V_\delta\|ds.
\end{split}
\end{equation}
>From Lemma~\ref{lemma25} it follows that there exists an 
$a(t)$ such that 
\begin{equation}
\label{eq176}
\frac{1}{2}a(t)\|V_\delta(t)\| \ge  e^{-\varphi(t)}\int_0^te^{\varphi(s)}|\dot{a}| \|V_\delta(s)\|ds.
\end{equation}
For example, one can choose 
\begin{equation}
\label{eq177}
a(t)=\frac{d}{(c+t)^b}, \quad b\in (0,\frac{1}{2}],\quad d c^{1-b}\ge 6b,
\end{equation}
where 
$d,c>0$.
Moreover, one can always choose $u_0$ such that 
\begin{equation}
\label{eq178}
h(0)=\|F(u_0) + a(0)u_0 -f_\delta\| \le \frac{1}{4} a(0)\|V_\delta(0)\|,
\end{equation}
because the equation 
\begin{equation}
\label{eq179}
F(u_0) + a(0)u_0 -f_\delta=0
\end{equation}
is uniquely solvable for any $f_\delta \in H$ if $a(0)>0$ and $F$ is 
monotone.

If \eqref{eq178} holds, then
\begin{equation}
\label{eq180}
h(0)e^{-\varphi(t)}\le \frac{1}{4}a(0)\|V_\delta(0)\|e^{-\varphi(t)},\qquad 
t\ge 0.
\end{equation}
If \eqref{eq177} holds, $c\ge1$  and $2b\le d$, then it follows that 
\begin{equation}
\label{eq181}
e^{-\varphi(t)}a(0)\le a(t).
\end{equation}
Indeed, inequality $a(0)\le a(t)e^{\varphi(t)}$ is obviously true for 
$t=0$, and $\big(a(t)e^{\varphi(t)}\big)'_t\geq 0$, provided that $c\geq 
1$ and $2b\le d$. 

Inequalities \eqref{eq180} and \eqref{eq181} imply
\begin{equation}
\label{eq182}
e^{-\varphi(t)}h(0) 
\le \frac{1}{4} a(t)\|V_\delta(0)\|
\le \frac{1}{4} a(t)\|V_\delta(t)\|,\quad t\ge 0,
\end{equation}
where we have used the inequality $\|V_\delta(t)\|\le \|V_\delta(t')\|$ for $t\le t'$, 
established in Lemma~\ref{lemma17}.
>From \eqref{eq160} and \eqref{eq164}--\eqref{eq182}, one gets
\begin{equation}
\label{eq183}
C\delta^\zeta = \|F(u_\delta(t_\delta))-f_\delta\|\ge \frac{1}{4}a(t_\delta)\|V_\delta(t_\delta)\|.
\end{equation}
Thus,
\begin{equation}
\label{eq184}
\lim_{\delta\to0}a(t_\delta)\|V_\delta(t_\delta)\|\le 
\lim_{\delta\to0}4C\delta^\zeta = 0.
\end{equation}
Since $\|V_\delta(t)\|$ is increasing, this implies 
$\lim_{\delta\to0}a(t_\delta)=0$. 
Since $0<a(t)\searrow 0$, it follows that  
\eqref{eq164} holds. 

>From the triangle inequality and inequalities \eqref{eq155} and 
\eqref{eq337} one obtains:
\begin{equation}
\label{eq185}
\begin{split}
\|u_\delta(t_\delta) - y\| &\le \|u_{\delta}(t_\delta) - V_{\delta}\| + 
\|V(t_\delta) - V_\delta(t_\delta)\| + \|V(t_\delta) - y\|\\
&\le \frac{a(t_\delta)}{\lambda} + \frac{\delta}{a(t_\delta)} + \|V(t_\delta)-y\|,
\end{split}
\end{equation}
where $V(t):=V_\delta(t)|_{\delta=0}$ and $V_\delta(t)$ solves \eqref{eq41.2}. 
>From \eqref{eq163}, \eqref{eq164}, inequality \eqref{eq185} and Lemma~\ref{lemma16}, one obtains
\eqref{eq146}. Theorem~\ref{theorem9} is proved.
\end{proof}

By the arguments, similar to the ones in the proof of 
Theorem~\ref{theorem10}--\ref{theorem12}
or in Remark~\ref{hehe}, one can show that:
the trajectory $u_\delta(t)$ remains in the ball
$B(u_0,R):=\{u: \|u-u_0\|<R\}$ for all $t\leq t_\delta$, where $R$ does
not depend on $\delta$ as $\delta\to 0$.

\subsection{Proofs of convergence of the iterative schemes}

\subsubsection{Proof of Theorem~\ref{theorem10}}

\begin{proof}[Proof of Theorem~\ref{theorem10}]
Denote 
\begin{equation}
\label{eq192}
C:=\frac{C_1+1}{2}.
\end{equation}
Let 
\begin{equation}
\label{eq193}
z_n:=u_n-V_n,\quad g_n:=\|z_n\|.
\end{equation}
We use Taylor's formula and get:
\begin{equation}
\label{eq194}
F(u_n)-F(V_n)+a_nz_n=A_{a_n} z_n+ K_n, \quad \|K_n\| \le\frac{M_2}{2}\|z_n\|^2,
\end{equation}
where $K_n:=F(u_n)-F(V_n)-F'(u_n)z_n$ and $M_2$ is the constant from \eqref{eq44}.
>From \eqref{eq187} and \eqref{eq194} one obtains
\begin{equation}
\label{eq195}
z_{n+1} = z_n - z_n - A_n^{-1}K(z_n) - (V_{n+1}-V_{n}).
\end{equation}
>From \eqref{eq195}, \eqref{eq194}, and the estimate $\|A_n^{-1}\|\le\frac{1}{a_n}$, one gets
\begin{equation}
\label{eq196}
g_{n+1} \le \frac{M_2g_n^2}{2a_n} + \|V_{n+1}-V_n\|.
\end{equation}
There exists a unique $n_0$ such that
\begin{equation}
\label{eq197}
\frac{\delta}{a_{n_0+1}}> \frac{1}{C-1}\|y\|\ge \frac{\delta}{a_{n_0}},\qquad C>1.
\end{equation} 
It follows from \eqref{eq197} and Lemma~\ref{lemma22} that inequalities \eqref{eq319} and \eqref{eq321} hold.

Inequalities \eqref{eq196} and \eqref{eq321} imply
\begin{equation}
\label{eq198}
g_{n+1}\le \frac{c_0}{a_n}g_n^2+\frac{a_n-a_{n+1}}{a_{n+1}}c_1,
\quad c_0=\frac{M_2}{2},\quad c_1=\|y\|\bigg{(}1+\frac{2}{C-1}\bigg{)},
\end{equation}
for all $n\le n_0+1$. 

Let us show by induction that 
\begin{equation}
\label{eq199}
g_n<\frac{a_n}{\lambda},\qquad 0\le n\le n_0+1.
\end{equation}
Inequality \eqref{eq199} holds for $n=0$ by Remark~\ref{remark8} (see \eqref{eq360}). Suppose \eqref{eq199} holds for some $n\ge 0$. 
>From \eqref{eq198}, \eqref{eq199} and \eqref{eq346}, one gets
\begin{equation}
\label{eq200}
\begin{split}
g_{n+1}&\le \frac{c_0}{a_n}\bigg{(}\frac{a_n}{\lambda}\bigg{)}^2 + \frac{a_n-a_{n+1}}{a_{n+1}}c_1\\
&= \frac{c_0 a_n}{\lambda^2} + \frac{a_n-a_{n+1}}{a_{n+1}}c_1\\
&\le \frac{a_{n+1}}{\lambda}.
\end{split}
\end{equation}
Thus, by induction, inequality \eqref{eq199} holds for all $n$ in the region $0\le n\le n_0+1$.

>From inequality \eqref{eq338} one has $\|V_n\| \le \|y\|+\frac{\delta}{a_n}$. 
This and the triangle inequality imply 
\begin{equation}
\label{eq201}
\|u_0-u_n\| \le \|u_0\|+ \|z_n\|+ \|V_n\|\le \|u_0\|+\|z_n\|+ \|y\|+\frac{\delta}{a_n}.
\end{equation}
Inequalities \eqref{eq320}, \eqref{eq199},
and \eqref{eq201} guarantee that the sequence $u_n$, generated by the 
iterative process \eqref{eq187}, remains
in the ball $B(u_0,R)$ for all $n\le n_0+1$, where 
$R\le \frac{a_0}{\lambda}+\|u_0\|+\|y\|+ \frac{\delta}{a_n}$.
This inequality and the estimate \eqref{eq199} imply that the sequence 
$u_n$, $n\le 
n_0+1,$ stays in the ball $B(u_0,R)$,
where 
\begin{equation}
\label{eq202}
R\le \frac{a_0}{\lambda}+ \|u_0\|+\|y\|+ \|y\|\frac{C+1}{C-1}.
\end{equation}
By Remark~\ref{remark7}, one can choose $a_0$ and $\lambda$ so that 
$\frac{a_0}{\lambda}$
is uniformly bounded as $\delta \to 0$ even if $M_1(R)\to\infty$ as 
$R\to\infty$ at an arbitrary fast rate.
Thus, the sequence $u_n$ stays in the ball $B(u_0,R)$ for $n\leq n_0+1$
when $\delta\to 0$. An upper bound on
$R$ is given above. It does not depend on $\delta$ as 
$\delta\to 0$.

One has:
\begin{equation}
\label{eq203}
\begin{split}
\|F(u_n)-f_\delta\|\le& \|F(u_n)-F(V_n)\|+\|F(V_n)-f_\delta\|\\
\le& M_1g_n+\|F(V_n)-f_\delta\|\\
\le& \frac{M_1a_n}{\lambda} + \|F(V_n)-f_\delta\|,\qquad \forall n\le n_0+1,
\end{split}
\end{equation}
where \eqref{eq199} was used and $M_1$ is the constant from \eqref{eq44}. 
By Lemma~\ref{lemma22} one gets 
\begin{equation}
\label{eq204}
\|F(V_{n_0+1})-f_\delta\|\le C\delta.
\end{equation}
>From \eqref{eq345}, \eqref{eq203}, \eqref{eq204}, 
the relation \eqref{eq197}, 
and the definition $C_1=2C-1$ (see \eqref{eq192}), one concludes that
\begin{equation}
\label{eq205}
\begin{split}
\|F(u_{n_0+1})-f_\delta\| 
\le& \frac{M_1a_{n_0+1}}{\lambda} + C\delta \\
\le& \frac{M_1\delta (C-1)}{\lambda\|y\|} + C\delta\le (2C-1)\delta=C_1\delta.
\end{split}
\end{equation}
{\it Thus, if 
\begin{equation}
\label{eq206}
\|F(u_0)-f_\delta\|> C_1\delta^\gamma,\quad 0<\gamma\le 1,
\end{equation}
then one concludes from \eqref{eq205} that there exists 
$n_\delta$, $0<n_\delta \le 
n_0+1,$ such that
\begin{equation}
\label{eq207}
\|F(u_{n_\delta})-f_\delta\| \le C_1\delta^\gamma < \|F(u_{n})-f_\delta\|,\quad 0\le n< n_\delta,
\end{equation}
for any given $\gamma\in (0,1]$, and any fixed $C_1>1$.}

Let us prove \eqref{eq189}. If $n>0$ is fixed, then $u_{\delta,n}$ is a
continuous function of $f_\delta$. Denote
\begin{equation}
\label{eq208}
\tilde{u}_N=\lim_{\delta\to 0}u_{\delta,N},
\end{equation}
where $N<\infty$ is a cluster point of $n_{\delta_m}$, so that there exists a subsequence of $n_{\delta_m}$,
which we denote by $n_{m}$, such that
\begin{equation}
\label{eq209}
\lim_{m\to\infty}n_{m} = N.
\end{equation}
>From \eqref{eq208} and the continuity of $F$, one obtains:
\begin{equation}
\label{eq210}
\|F(\tilde{u}_N)-f\| = \lim_{m\to\infty}\|F(u_{n_{\delta_m}})-f_{\delta_m}\|\le \lim_{m\to \infty}C_1\delta_m^\gamma = 0.
\end{equation}
Thus, $\tilde{u}_N$ is a solution to the equation $F(u)=f$, and \eqref{eq189} is proved.

{\it Let us prove \eqref{eq191} assuming that \eqref{eq190} holds.} From \eqref{eq188} and \eqref{eq203}  with $n=n_\delta-1$, and from \eqref{eq207}, one gets
\begin{equation}
\label{eq211}
\begin{split}
C_1\delta^\gamma &\le M_1 \frac{a_{n_\delta-1}}{\lambda} + a_{n_\delta-1}\|V_{n_\delta-1}\|
\le M_1 \frac{a_{n_\delta-1}}{\lambda} + \|y\|a_{n_\delta-1}+\delta.
\end{split}
\end{equation}
If $0<\delta<1$ and $\delta$ is sufficiently small, then
\begin{equation}
\label{eq212}
\tilde{C}\delta^\gamma \le a_{n_\delta-1} \bigg{(}\frac{M_1}{\lambda}+\|y\|\bigg{)},\quad \tilde{C}>0,
\end{equation}
where $\tilde{C}$ is a constant. 
Therefore, by \eqref{eq212}, 
\begin{equation}
\label{eq213}
\lim_{\delta\to 0} \frac{\delta}{2a_{n_\delta}}\le
\lim_{\delta\to 0} \frac{\delta}{a_{n_\delta-1}}\le
\lim_{\delta\to 0} \frac{\delta^{1-\gamma}}{\tilde{C}}\bigg{(}\frac{M_1}{\lambda}+\|y\|\bigg{)}
=0,\quad 0<\gamma<1.
\end{equation}
In particular, for $\delta=\delta_m$, one gets
\begin{equation}
\label{eq214}
\lim_{\delta_m\to 0} \frac{\delta_m}{a_{n_{\delta_m}}} = 0.
\end{equation}

>From the triangle inequality, inequalities \eqref{eq337} 
and \eqref{eq199}, one obtains
\begin{equation}
\label{eq215}
\begin{split}
\|u_{n_{\delta_m}} - y\| &\le \|u_{n_{\delta_m}} - V_{n_{\delta_m}}\| + 
\|V_{n_{\delta_m}} - V_{{n_{\delta_m}},0}\| + \|V_{{n_{\delta_m}},0}-y\|\\
&\le \frac{a_{n_{\delta_m}}}{\lambda} + \frac{\delta_m}{a_{n_{\delta_m}}} + \|V_{{n_{\delta_m}},0}-y\|.
\end{split}
\end{equation}
Recall that $V_{n,0}=\tilde{V}_{a_n}$ (cf. \eqref{eq186} and \eqref{eq41.2}). 
>From \eqref{eq190}, \eqref{eq214}, inequality \eqref{eq215} and Lemma~\ref{lemma16}, one obtains
\eqref{eq191}. Theorem~\ref{theorem10} is proved.
\end{proof}

\subsubsection{Proof of Theorem~\ref{theorem11}}

\begin{proof}[Proof of Theorem~\ref{theorem11}]
Denote 
\begin{equation}
\label{eq223}
C:=\frac{C_1+1}{2}.
\end{equation}
Let 
\begin{equation}
\label{eq224}
z_n:=u_n-V_n,\quad g_n:=\|z_n\|.
\end{equation}
We use Taylor's formula and get:
\begin{equation}
\label{eq225}
F(u_n)-F(V_n)+a_nz_n=A_{n} z_n+ K_n, \quad \|K_n\| \le\frac{M_2}{2}\|z_n\|^2,
\end{equation}
where $K_n:=F(u_n)-F(V_n)-F'(u_n)z_n$ and $M_2$ is the constant 
from \eqref{eq44}.
>From \eqref{eq216} and \eqref{eq225} one obtains
\begin{equation}
\label{eq226}
z_{n+1} = z_n - \alpha_n A_n^*A_nz_n - \alpha_n A_n^*K(z_n) - (V_{n+1}-V_{n}).
\end{equation}
>From \eqref{eq226}, \eqref{eq225}, \eqref{eq218}, and the 
estimate $\|A_n\|\le M_1+a_n$, one gets
\begin{equation}
\label{eq227}
\begin{split}
g_{n+1}   &\le g_n\|1 - \alpha_n A_n^*A_n\|+ \frac{\alpha_n M_2(M_1+a_n)}{2}g_n^2 + \|V_{n+1}-V_n\|\\
          &\le g_n(1-\alpha_n a_n^2)+\frac{\alpha_n M_2(M_1+a_n)}{2}g_n^2 
+ \|V_{n+1}-V_n\|.
\end{split}
\end{equation}
Since $0<a_n\searrow 0$, for any fixed $\delta>0$ there exists $n_0$ such that
\begin{equation}
\label{eq228}
\frac{\delta}{a_{n_0+1}}> \frac{1}{C-1}\|y\|\ge \frac{\delta}{a_{n_0}},\qquad C>1.
\end{equation} 
This and Lemma~\ref{lemma22} imply that inequalities \eqref{eq319}--\eqref{eq321} hold. 

Inequalities \eqref{eq227} and \eqref{eq321} imply
\begin{equation}
\label{eq229}
g_{n+1}\le (1-\alpha_n a_n^2)g_n + \alpha_n c_0(M_1+a_n)g_n^2+\frac{a_n-a_{n+1}}{a_{n+1}}c_1,\qquad \forall \, n\le n_0+1,
\end{equation}
where the constants $c_0$ and $c_1$ are defined in \eqref{eq103}.


Let us show by induction that 
\begin{equation}
\label{eq230}
g_n<\frac{a_n^2}{\lambda},\qquad 0\le n\le n_0+1.
\end{equation}
Inequality \eqref{eq230} holds for $n=0$ by Remark~\ref{remark8} (see \eqref{eq361}). Suppose \eqref{eq230} holds for some $n\ge 0$. 
>From \eqref{eq229}, \eqref{eq230} and \eqref{eq352}, one gets
\begin{equation}
\label{eq231}
\begin{split}
g_{n+1}&\le (1-\alpha_n a_n^2)\frac{a_n^2}{\lambda}+ 
\alpha_n c_0(M_1+a_n)\bigg{(}\frac{a_n^2}{\lambda}\bigg{)}^2 + \frac{a_n-a_{n+1}}{a_{n+1}}c_1\\
&=\frac{a_n^4}{\lambda}\bigg{(}\frac{\alpha_n c_0(M_1+a_n)}{\lambda}-\alpha_n\bigg{)}+ 
\frac{a_n^2}{\lambda} + \frac{a_n-a_{n+1}}{a_{n+1}}c_1\\
&\le -\frac{\alpha_n a_n^4}{2\lambda}+ 
\frac{a_n^2}{\lambda} + \frac{a_n-a_{n+1}}{a_{n+1}}c_1\\
&\le \frac{a_{n+1}^2}{\lambda}.
\end{split}
\end{equation}
Thus, by induction, inequality \eqref{eq230} holds for all $n$ in the region $0\le n\le n_0+1$.

>From \eqref{eq338} one has $\|V_n\| \le \|y\|+\frac{\delta}{a_n}$. 
This and the triangle inequality imply 
\begin{equation}
\label{eq232}
\|u_0-u_n\| \le \|u_0\|+ \|z_n\|+ \|V_n\|\le \|u_0\|+\|z_n\|+ \|y\|+\frac{\delta}{a_n}.
\end{equation}
Inequalities \eqref{eq320}, \eqref{eq230},
and \eqref{eq232} guarantee that the sequence $u_n$, generated by the 
iterative process \eqref{eq216}, remains
in the ball $B(u_0,R)$ for all $n\le n_0+1$, where 
$R\le \frac{a_0}{\lambda}+\|u_0\|+\|y\|+ \frac{\delta}{a_n}$.
This inequality and the estimate \eqref{eq230} imply that the sequence 
$u_n$, $n\le 
n_0+1,$ stays in the ball $B(u_0,R)$,
where 
\begin{equation}
\label{eq233}
R\le \frac{a_0}{\lambda}+ \|u_0\|+\|y\|+ \|y\|\frac{C+1}{C-1}.
\end{equation}
By Remark~\ref{remark7}, one can choose $a_0$ and $\lambda$ so that 
$\frac{a_0}{\lambda}$
is uniformly bounded as $\delta \to 0$ even if $M_1(R)\to\infty$ as 
$R\to\infty$ at an arbitrary fast rate.
Thus, the sequence $u_n$ stays in the ball $B(u_0,R)$ for $n\leq n_0+1$
when $\delta\to 0$. An upper bound on
$R$ is given above. It does not depend on $\delta$ as 
$\delta\to 0$.

One has:
\begin{equation}
\label{eq234}
\begin{split}
\|F(u_n)-f_\delta\|\le& \|F(u_n)-F(V_n)\|+\|F(V_n)-f_\delta\|\\
\le& M_1g_n+\|F(V_n)-f_\delta\|\\
\le& \frac{M_1a_n^2}{\lambda} + \|F(V_n)-f_\delta\|,\qquad \forall n\le n_0+1,
\end{split}
\end{equation}
where \eqref{eq230} was used and $M_1$ is the constant from \eqref{eq44}. 
By Lemma~\ref{lemma22} one gets 
\begin{equation}
\label{eq235}
\|F(V_{n_0+1})-f_\delta\| \le C\delta.
\end{equation}
>From \eqref{eq350}, \eqref{eq234}, \eqref{eq235}, 
the relation \eqref{eq228}, 
and the definition $C_1=2C-1$ (see \eqref{eq223}), one concludes that
\begin{equation}
\label{eq236}
\begin{split}
\|F(u_{n_0+1})-f_\delta\| 
\le& \frac{M_1a_{n_0+1}^2}{\lambda} + C\delta \\
\le& \frac{M_1\delta (C-1)}{\lambda\|y\|} + C\delta\le (2C-1)\delta=C_1\delta.
\end{split}
\end{equation}
{\it Thus, if 
\begin{equation}
\label{eq237}
\|F(u_0)-f_\delta\|> C_1\delta^\zeta,\quad 0<\zeta\le 1,
\end{equation}
then one concludes from \eqref{eq236} that there exists 
$n_\delta$, $0<n_\delta \le 
n_0+1,$ such that
\begin{equation}
\label{eq238}
\|F(u_{n_\delta})-f_\delta\| \le C_1\delta^\zeta < \|F(u_{n})-f_\delta\|,\quad 0\le n< n_\delta,
\end{equation}
for any given $\zeta\in (0,1]$, and any fixed $C_1>1$.}

{\it Let us prove \eqref{eq220}.}

 If $n>0$ is fixed, then 
$u_{\delta,n}$ is a
continuous function of $f_\delta$. Denote
\begin{equation}
\label{eq239}
\tilde{u}:=\tilde{u}_N=\lim_{\delta\to 0}u_{\delta,n_{m_j}},
\end{equation}
where 
\begin{equation}
\label{eq240}
\lim_{j\to\infty}n_{m_j} = N.
\end{equation}
>From \eqref{eq239} and the continuity of $F$, one obtains:
\begin{equation}
\label{eq241}
\|F(\tilde{u})-f_\delta\| = \lim_{j\to\infty}\|F(u_{n_{m_j}})-f_\delta\|\le \lim_{\delta\to 0}C_1\delta^\zeta = 0.
\end{equation}
Thus, $\tilde{u}$ is a solution to the equation $F(u)=f$, and \eqref{eq220} is proved.

{\it Let us prove \eqref{eq222} assuming that \eqref{eq221} holds.} 

>From \eqref{eq219} and \eqref{eq234}  with $n=n_\delta-1$, 
and from \eqref{eq238}, one gets
\begin{equation}
\label{eq242}
\begin{split}
C_1\delta^\zeta &\le M_1 \frac{a_{n_\delta-1}^2}{\lambda} + 
a_{n_\delta-1}\|V_{n_\delta-1}\|
\le M_1 \frac{a_{n_\delta-1}^2}{\lambda} + \|y\|a_{n_\delta-1}+\delta.
\end{split}
\end{equation}
If $\delta>0$ is sufficiently small, then the above equation implies
\begin{equation}
\label{eq243}
\tilde{C}\delta^\zeta \le a_{n_\delta-1} \bigg{(}\frac{M_1a_0}{\lambda}+\|y\|\bigg{)},\quad \tilde{C}>0,
\end{equation}
where $\tilde{C}<C_1$ is a constant, and the inequality $ 
a^2_{n_\delta-1}\le a_{n_\delta-1}a_0$ was used. 
Therefore, by \eqref{eq348}, 
\begin{equation}
\label{eq244}
\lim_{\delta\to 0} \frac{\delta}{2a_{n_\delta}}\le
\lim_{\delta\to 0} \frac{\delta}{a_{n_\delta-1}}\le
\lim_{\delta\to 0} \frac{\delta^{1-\zeta}}{\tilde{C}}\bigg{(}\frac{M_1a_0}{\lambda}+\|y\|\bigg{)}
=0,\quad 0<\zeta<1.
\end{equation}
In particular, for $\delta=\delta_m$, one gets
\begin{equation}
\label{eq245}
\lim_{\delta_m\to 0} \frac{\delta_m}{a_{n_m}} = 0.
\end{equation}
>From the triangle inequality and inequalities \eqref{eq337} and \eqref{eq230} one obtains
\begin{equation}
\label{eq246}
\begin{split}
\|u_{n_m} - y\| &\le \|u_{n_m} - V_{n_m}\| + 
\|V_n - V_{{n_m},0}\| + \|V_{{n_m},0}-y\|\\
&\le \frac{a^2_{n_m}}{\lambda} + \frac{\delta_m}{a_{n_m}} + \|V_{{n_m},0}-y\|.
\end{split}
\end{equation}
Recall that $V_{n,0}=\tilde{V}_{a_n}$ (cf. \eqref{eq186} and \eqref{eq41.2}). 
>From \eqref{eq221}, \eqref{eq245}, inequality \eqref{eq246} and Lemma~\ref{lemma16}, one obtains
\eqref{eq222}. Theorem~\ref{theorem11} is proved.
\end{proof}

\subsubsection{Proof of Theorem~\ref{theorem12}}

\begin{proof}
Denote 
\begin{equation}
\label{eq254}
C:=\frac{C_1+1}{2}.
\end{equation}
Let 
\begin{equation}
\label{eq255}
z_n:=u_n-V_n,\quad g_n:=\|z_n\|.
\end{equation}
One has
\begin{equation}
\label{eq256}
F(u_n) - F(V_n)=J_nz _n,\qquad J_n=\int_0^1 F'(u_0+\xi z_n)d\xi.
\end{equation}
Since $F'(u)\ge 0,\,\forall u\in H$ and $\|F'(u)\|\le M_1,\forall u\in B(u_0,R)$,  it follows that $J_n\ge 0$
and $\|J_n\|\le M_1$. 
>From \eqref{eq247} and \eqref{eq256} one obtains
\begin{equation}
\label{eq257}
\begin{split}
z_{n+1} &= z_n - \alpha_n [F(u_n) - F(V_n) + a_n z_n] - (V_{n+1}-V_{n})\\
& = (1-\alpha_n (J_n+a_n))z_n - (V_{n+1}-V_{n}).
\end{split}
\end{equation}
>From \eqref{eq257} and \eqref{eq249}, one gets
\begin{equation}
\label{eq258}
\begin{split}
g_{n+1}   &\le g_n \|1-\alpha_n(J_n + a_n)\| + \|V_{n+1}-V_n\|\\
&\le g_n(1-\alpha_n a_n) + \|V_{n+1}-V_n\|.
\end{split}
\end{equation}
Since $0<a_n\searrow 0$, for any fixed $\delta>0$ there exists $n_0$ such that
\begin{equation}
\label{eq259}
\frac{\delta}{a_{n_0+1}}> \frac{1}{C-1}\|y\|\ge \frac{\delta}{a_{n_0}},\qquad C>1.
\end{equation} 
This and Lemma~\ref{lemma22} imply that inequalities \eqref{eq319}--\eqref{eq321} hold. 

Inequalities \eqref{eq258} and \eqref{eq321} imply
\begin{equation}
\label{eq260}
g_{n+1}\le (1-\alpha_n a_n)g_n +\frac{a_n-a_{n+1}}{a_{n+1}}c_1,\qquad \forall \, n\le n_0+1,
\end{equation}
where the constant $c_1$ is defined in \eqref{eq152}.


Let us show by induction that 
\begin{equation}
\label{eq261}
g_n<\frac{a_n}{\lambda},\qquad 0\le n\le n_0+1.
\end{equation}
Inequality \eqref{eq261} holds for $n=0$ by Remark~\ref{remark8} (see \eqref{eq360}). Suppose \eqref{eq261} holds for some $n\ge 0$. 
>From \eqref{eq260}, \eqref{eq261} and \eqref{eq356}, one gets
\begin{equation}
\label{eq262}
\begin{split}
g_{n+1}&\le (1-\alpha_n a_n)\frac{a_n}{\lambda}+ 
 \frac{a_n-a_{n+1}}{a_{n+1}}c_1\\
&= -\frac{\alpha_n a_n^2}{\lambda}+ 
\frac{a_n}{\lambda} + \frac{a_n-a_{n+1}}{a_{n+1}}c_1\\
&\le \frac{a_{n+1}}{\lambda}.
\end{split}
\end{equation}
Thus, by induction, inequality \eqref{eq261} holds for all $n$ in the region $0\le n\le n_0+1$.

>From \eqref{eq338} one has $\|V_n\| \le \|y\|+\frac{\delta}{a_n}$. 
This and the triangle inequality imply 
\begin{equation}
\label{eq263}
\|u_0-u_n\| \le \|u_0\|+ \|z_n\|+ \|V_n\|\le \|u_0\|+\|z_n\|+ \|y\|+\frac{\delta}{a_n}.
\end{equation}
Inequalities \eqref{eq320}, \eqref{eq261},
and \eqref{eq263} guarantee that the sequence $u_n$, generated by the 
iterative process \eqref{eq247}, remains
in the ball $B(u_0,R)$ for all $n\le n_0+1$, where 
$R\le \frac{a_0}{\lambda}+\|u_0\|+\|y\|+ \frac{\delta}{a_n}$.
This inequality and the estimate \eqref{eq261} imply that the sequence 
$u_n$, $n\le 
n_0+1,$ stays in the ball $B(u_0,R)$,
where 
\begin{equation}
\label{eq264}
R\le \frac{a_0}{\lambda}+ \|u_0\|+\|y\|+ \|y\|\frac{C+1}{C-1}.
\end{equation}
By Remark~\ref{remark7}, one can choose $a_0$ and $\lambda$ so that 
$\frac{a_0}{\lambda}$
is uniformly bounded as $\delta \to 0$ even if $M_1(R)\to\infty$ as 
$R\to\infty$ at an arbitrary fast rate.
Thus, the sequence $u_n$ stays in the ball $B(u_0,R)$ for $n\leq n_0+1$
when $\delta\to 0$. An upper bound on
$R$ is given above. It does not depend on $\delta$ as 
$\delta\to 0$.

One has:
\begin{equation}
\label{eq265}
\begin{split}
\|F(u_n)-f_\delta\|\le& \|F(u_n)-F(V_n)\|+\|F(V_n)-f_\delta\|\\
\le& M_1g_n+\|F(V_n)-f_\delta\|\\
\le& \frac{M_1a_n}{\lambda} + \|F(V_n)-f_\delta\|,\qquad \forall n\le n_0+1,
\end{split}
\end{equation}
where \eqref{eq261} was used and $M_1$ is the constant from \eqref{eq44}. 
Since $\|F(V_n)-f_\delta\|$ is decreasing, by Lemma~\ref{lemma17}, and
$n_1\le n_0$, one gets 
\begin{equation}
\label{eq266}
\|F(V_{n_0+1})-f_\delta\|\le \|F(V_{n_1+1})-f_\delta\| \le C\delta.
\end{equation}
>From \eqref{eq355}, \eqref{eq265}, \eqref{eq266}, 
the relation \eqref{eq259}, 
and the definition $C_1=2C-1$ (see \eqref{eq254}), one concludes that
\begin{equation}
\label{eq267}
\begin{split}
\|F(u_{n_0+1})-f_\delta\| 
\le& \frac{M_1a_{n_0+1}}{\lambda} + C\delta \\
\le& \frac{M_1\delta (C-1)}{\lambda\|y\|} + C\delta\le (2C-1)\delta=C_1\delta.
\end{split}
\end{equation}
{\it Thus, if 
\begin{equation}
\label{eq268}
\|F(u_0)-f_\delta\|> C_1\delta^\zeta,\quad 0<\zeta\le 1,
\end{equation}
then one concludes from \eqref{eq267} that there exists 
$n_\delta$, $0<n_\delta \le 
n_0+1,$ such that
\begin{equation}
\label{eq269}
\|F(u_{n_\delta})-f_\delta\| \le C_1\delta^\zeta < \|F(u_{n})-f_\delta\|,\quad 0\le n< n_\delta,
\end{equation}
for any given $\zeta\in (0,1]$, and any fixed $C_1>1$.}

{\it Let us prove \eqref{eq251}.}

 If $n>0$ is fixed, then 
$u_{\delta,n}$ is a
continuous function of $f_\delta$. Denote
\begin{equation}
\label{eq270}
\tilde{u}:=\tilde{u}_N=\lim_{\delta\to 0}u_{\delta,n_{m_j}},
\end{equation}
where 
\begin{equation}
\label{eq271}
\lim_{j\to\infty}n_{m_j} = N.
\end{equation}
>From \eqref{eq270} and the continuity of $F$, one obtains:
\begin{equation}
\label{eq272}
\|F(\tilde{u})-f_\delta\| = \lim_{j\to\infty}\|F(u_{n_{m_j}})-f_\delta\|\le \lim_{\delta\to 0}C_1\delta^\zeta = 0.
\end{equation}
Thus, $\tilde{u}$ is a solution to the equation $F(u)=f$, and \eqref{eq251} is proved.

{\it Let us prove \eqref{eq253} assuming that \eqref{eq252} holds.} 

>From \eqref{eq250} and \eqref{eq265}  with $n=n_\delta-1$, 
and from \eqref{eq269}, one gets
\begin{equation}
\label{eq273}
\begin{split}
C_1\delta^\zeta &\le M_1 \frac{a_{n_\delta-1}}{\lambda} + 
a_{n_\delta-1}\|V_{n_\delta-1}\|
\le M_1 \frac{a_{n_\delta-1}}{\lambda} + \|y\|a_{n_\delta-1}+\delta.
\end{split}
\end{equation}
If $\delta>0$ is sufficiently small, then the above equation implies
\begin{equation}
\label{eq274}
\tilde{C}\delta^\zeta \le a_{n_\delta-1} \bigg{(}\frac{M_1}{\lambda}+\|y\|\bigg{)},\quad \tilde{C}>0,
\end{equation}
where $\tilde{C}<C_1$ is a constant. 
Therefore, by \eqref{eq353}, 
\begin{equation}
\label{eq275}
\lim_{\delta\to 0} \frac{\delta}{2a_{n_\delta}}\le
\lim_{\delta\to 0} \frac{\delta}{a_{n_\delta-1}}\le
\lim_{\delta\to 0} \frac{\delta^{1-\zeta}}{\tilde{C}}\bigg{(}\frac{M_1}{\lambda}+\|y\|\bigg{)}
=0,\quad 0<\zeta<1.
\end{equation}
In particular, for $\delta=\delta_m$, one gets
\begin{equation}
\label{eq276}
\lim_{\delta_m\to 0} \frac{\delta_m}{a_{n_m}} = 0.
\end{equation}
>From the triangle inequality, inequalities \eqref{eq337} and \eqref{eq261}, one obtains
\begin{equation}
\label{eq277}
\begin{split}
\|u_{n_m} - y\| &\le \|u_{n_m} - V_{n_m}\| + 
\|V_n - V_{{n_m},0}\| + \|V_{{n_m},0}-y\|\\
&\le \frac{a_{n_m}}{\lambda} + \frac{\delta_m}{a_{n_m}} + \|V_{{n_m},0}-y\|.
\end{split}
\end{equation}
Recall that $V_{n,0}=\tilde{V}_{a_n}$ (cf. \eqref{eq186} and \eqref{eq41.2}). 
>From \eqref{eq252}, \eqref{eq276}, inequality \eqref{eq277} and Lemma~\ref{lemma16}, one obtains
\eqref{eq253}. Theorem~\ref{theorem12} is proved.
\end{proof}

\subsection{Proofs of the nonlinear inequalities}

\begin{proof}[Proof of Theorem~\ref{theorem13}]
Denote $w(t):=g(t)e^{\int_{\tau_0}^t\gamma(s)ds}$. Then inequality 
\eqref{eq278} takes the form
\begin{equation}
\label{558eq6}
\dot{w}(t) \le a(t)w^p(t) + b(t),\qquad w(\tau_0)=g(\tau_0):= g_0,
\end{equation}
where
\begin{equation}
a(t):=\alpha(t) e^{(1-p)\int_{\tau_0}^t\gamma(s)ds},
\qquad b(t):=\beta(t) e^{\int_{\tau_0}^t\gamma(s)ds}.
\end{equation}
Denote 
\begin{equation}
\label{558eq8}
\eta(t) = \frac{e^{\int_{\tau_0}^t\gamma(s)ds}}{\mu(t)}.
\end{equation}
>From inequality \eqref{eq280} and relation \eqref{558eq8} one gets
\begin{equation}
\label{558eq9}
w(\tau_0)=g(\tau_0) < \frac{1}{\mu(\tau_0)}=\eta(\tau_0).
\end{equation}
It follows from the inequalities \eqref{eq337}, \eqref{558eq6} 
and \eqref{558eq9} that
\begin{equation}
\label{558eq10}
\begin{split}
\dot{w}(\tau_0) &\le \alpha(\tau_0)\frac{1}{\mu^{p}(\tau_0)}
+ \beta(\tau_0)
\le \frac{1}{\mu(\tau_0)}\bigg{[}\gamma -
\frac{\dot{\mu}(\tau_0)}{\mu(\tau_0)}\bigg{]}
 = \frac{d}{dt} \frac{e^{\int_{\tau_0}^t\gamma(s)ds}}
{\mu(t)}\bigg{|}_{t=\tau_0} = \dot{\eta}(\tau_0).
\end{split}
\end{equation}
>From the inequalities \eqref{558eq9} and \eqref{558eq10} it follows 
that there exists $\delta>0$ such that
\begin{equation}
w(t) < \eta(t),\qquad  \tau_0\le t \le \tau_0 + \delta.
\end{equation}
To continue the proof we need two Claims.

{\it Claim 1.}  {\it If} 
\begin{equation}
w(t) \le \eta(t),\qquad \forall t \in[ \tau_0, T],\quad T>\tau_0,
\end{equation}
{\it then} 
\begin{equation}
\dot{w}(t) \le \dot{\eta}(t),\qquad  \forall t \in[ \tau_0, T].
\end{equation}
{\it Proof of Claim 1.}
 
It follows from inequalities \eqref{eq279}, \eqref{558eq6} 
and the inequlity $w(T)\leq \eta(T)$,  that
\begin{equation}
\label{558eq14}
\begin{split}
\dot{w}(t) &\le e^{(1-p)\int_{\tau_0}^t\gamma(s)ds}\alpha(t)
\frac{e^{p\int_{\tau_0}^t\gamma(s)ds}}{\mu^{p}(t)}
+ \beta(t)e^{\int_{\tau_0}^t\gamma(s)ds}\\
&\le \frac{e^{\int_{\tau_0}^t\gamma(s)ds}}{\mu(t)}
\bigg{[}\gamma -\frac{\dot{\mu}(t)}{\mu(t)}\bigg{]}\\
& = \frac{d}{dt} \frac{e^{\int_{\tau_0}^t\gamma(s)ds}}
{\mu(t)}\bigg{|}_{t=t} = \dot{\eta}(t),\qquad \forall t\in [\tau_0,T].
\end{split}
\end{equation}
{\it Claim 1} is proved. 

Denote 
\begin{equation}
\label{558eq12}
T:=\sup \{\delta \in\mathbb{R}^+: w(t) < \eta(t),\, 
\forall t \in[\tau_0, \tau_0 + \delta]\}.
\end{equation}
{\it Claim 2.}  {\it One has $T=\infty$.} 

Claim 2 says that every nonnegative solution $g(t)$ to inequality (1),
satisfying assumption (3),  is defined 
globally.

{\it Proof of Claim 2.}

Assume the contrary, i.e.,  $T<\infty$. 
>From the definition of $T$ and the continuity of $w$ and $\eta$ one 
gets
\begin{equation}
\label{558eq337}
w(T) \le \eta(T).
\end{equation}
It follows from 
inequality \eqref{558eq337} and {\it Claim 1} that 
\begin{equation}
\label{558eq3388}
\dot{w}(t)\le \dot{\eta}(t),\qquad \forall t\in [\tau_0,T].
\end{equation}
This implies 
\begin{equation}
\label{558eq3389}
w(T)-w(\tau_0) = \int_{\tau_0}^T \dot{w}(s)ds \le \int_{\tau_0}^T \dot{\eta}(s)ds 
= \eta(T)-\eta(\tau_0).
\end{equation}
Since $w(\tau_0)<\eta(\tau_0)$ by assumption \eqref{eq280}, it follows from inequality \eqref{558eq3389} that
\begin{equation}
\label{558eq20}
w(T) < \eta(T).
\end{equation}
Inequality \eqref{558eq20} and inequality \eqref{558eq3388} with $t=T$ imply that
 there exists an $\epsilon>0$ such that
\begin{equation}
w(t) < \eta(t),\qquad  T\le t \le T + \epsilon.
\end{equation}
This contradicts the definition of $T$ in \eqref{558eq12}, and the 
contradiction proves the desired conclusion $T=\infty$. 

Claim 2 is proved. 

It follows from the definitions of $\eta(t)$ and $w(t)$ and from 
the relation $T=\infty$ that
\begin{equation}
g(t) = e^{-\int_{\tau_0}^t\gamma(s)ds} w(t)< 
e^{-\int_{\tau_0}^t\gamma(s)ds}\eta(t) = \frac{1}{\mu(t)},\qquad \forall t> \tau_0.
\end{equation}
Theorem~\ref{theorem13} is proved.
\end{proof}

\subsubsection{Proof of Theorem~\ref{548lemma16}}

\begin{proof}[Proof of Theorem~\ref{548lemma16}] 
Let us prove \eqref{558eq5} by induction. Inequality \eqref{558eq5} holds for $n=0$ by assumption \eqref{558eq2}. 
Suppose that \eqref{558eq5} holds for all $n\le m$. From inequalities 
\eqref{558eq1}, \eqref{558eq3}, and from the induction 
hypothesis 
$g_n\le\frac{1}{\mu_n}$, $n\le m$, one gets
\begin{equation}
\begin{split}
g_{m+1}&\le g_m(1-h_m\gamma_m) + \alpha_m h_m g_m^p + h_m\beta_m\\
\le& \frac{1}{\mu_m}(1-h_m\gamma_m)+ h_m\frac{\alpha_m}{\mu_m^p}+ h_m\beta_m\\
\le& \frac{1}{\mu_m}(1-h_m\gamma_m) + \frac{h_m}{\mu_m}\bigg{(}\gamma_m -
\frac{\mu_{m+1}-\mu_m}{\mu_m h_m}\bigg{)}\\
=& \frac{1}{\mu_m}-\frac{\mu_{m+1}-\mu_m}{\mu_m^2}\\
=& \frac{1}{\mu_{m+1}}- (\mu_{m+1}-\mu_m)\big{(}\frac{1}{\mu_m^2} - 
\frac{1}{\mu_m \mu_{m+1}} \big{)}\\
=& \frac{1}{\mu_{m+1}}- \frac{(\mu_{m+1}-\mu_m)^2}{\mu_n^2 \mu_{m+1}} 
\le\frac{1}{\mu_{m+1}}.
\end{split}
\end{equation}
Therefore, inequality \eqref{558eq5} holds for $n=m+1$. 
Thus, inequality \eqref{558eq5} holds for all $n\ge 0$ by induction. 
Theorem~\ref{548lemma16} is proved.
\end{proof}

\section{Applications of the nonlinear inequality \eqref{eq278}}

Here we only sketch the idea for many possible applications of this inequality
for a study of dynamical systems in a Hilbert space.

Let 
\begin{equation}
\label{eqr1}
\dot{u} = Au + h(t,u) + f(t),\qquad u(0)= u_0,\quad \dot{u}:=\frac{d u}{d t},\quad t\ge 0,
\end{equation}
where $A$ is a selfadjoint operator in a real Hilbert space,
$h(t,u)$ is a nonlinear operator in $H$, which is locally Lipschitz with 
respect to $u$
and H\"{o}lder-continuous with respect to $t\in \mathbb{R}_+:=[0,\infty)$,
and $f$ is a H\"{o}lder continuous function on $\mathbb{R}_+$ with values in $H$.

Assume that
\begin{equation}
\label{eqr2}
\Re \langle Au,u \rangle \le -\gamma(t) \langle Au,u \rangle,
\quad \Re \langle h(t,u),u \rangle \le \alpha(t)\|u\|^{1+p} \qquad 
\forall u\in D(A),
\end{equation}
where $\gamma(t)$ and $\alpha(t)$ are continuous functions on 
$\mathbb{R}_+$, $h(t,0)=0$, $p>1$ is a constant.

Our aim is to estimate the behavior of solutions to \eqref{eqr1} as $t\to\infty$,
in particular, to give sufficient conditions for a global existence of the unique solution
to \eqref{eqr1}. Our approach consists of a reduction of the problem to 
the inequality \eqref{eq278}
and an application of Theorem \ref{theorem13}.

Let $g(t):=\|u(t)\|$. Problem \eqref{eqr1} has a unique local solution under our assumptions.
Multiplying \eqref{eqr1} by $u$ from left, then from right, add, and use \eqref{eqr2} to get
\begin{equation}
\label{eqr3}
\dot{g}g\le -\gamma(t)g^2 + \alpha(t)g^{1+p} + \beta(t)g,\qquad \beta(t):=\Re \langle f(t),u\rangle.
\end{equation}
Since $g\ge 0$, one gets
\begin{equation}
\label{eqr4}
\dot{g} \le -\gamma(t)g + \alpha(t) g^p(t) + \beta(t).
\end{equation}

Now Theorem~\ref{theorem13} is applicable. This Theorem yields 
sufficient conditions \eqref{eq279} and \eqref{eq280} for the global
existence of the solution to \eqref{eqr1} and estimate \eqref{eq281} for the 
behavior of $\|u(t)\|$ as $t\to\infty$. 

The outlined scheme is widely applicable to stability problems, 
to semilinear parabolic problems, and to hyperbolic problems as well.
It yields some novel results. For instance, if the operator $A$
is a second-order elliptic operator with matrix $a_{ij}(x,t)$,
then Theorem~\ref{theorem13} allows one to treat degenerate problems,
namely, it allows, for example, the minimal eigenvalue  
$\lambda(x,t)$ of a selfadjoint matrix $a_{ij}(x,t)$ to depend
on time in such a way that $\min_{x}\lambda(x,t):=\lambda(t)\to 0$ as
$t\to \infty$ at a certain rate.

\section{A numerical experiment}

Let us present results of a numerical experiment. We solve 
nonlinear equation \eqref{eq1} with 
\begin{equation}
\label{5441eq41}
F(u):= B(u)+ \big{(}\arctan(u)\big{)}^3:=\int_0^1e^{-|x-y|}u(y)dy + \big{(}\arctan(u)\big{)}^3.
\end{equation}
Since the function $u\to \arctan^3u$ is increasing on $\mathbb{R}$, one 
has
\begin{equation}
\langle \big{(}\arctan(u)\big{)}^3 - \big{(}\arctan(v)\big{)}^3,u-v\rangle 
\ge 0,\qquad \forall\, u,v\in H.
\end{equation}
Moreover, 
\begin{equation}
\label{544eq77}
e^{-|x|} = \frac{1}{\pi}\int_{-\infty}^\infty \frac{e^{i\lambda x}}{1+\lambda^2} d\lambda.
\end{equation}
Therefore, $\langle B(u-v),u-v\rangle\ge0$, so 
\begin{equation}
\langle F(u-v),u-v\rangle\ge0,\qquad \forall\, u,v\in H.
\end{equation}

The Fr\'{e}chet derivative of $F$ is:
\begin{equation}
\label{544eq44}
F'(u)w = \frac{3\big{(}\arctan(u)\big{)}^2}{1+u^2}w + 
\int_0^1 e^{-|x-y|}w(y) dy.
\end{equation}
If $u(x)$ vanishes on a set of positive Lebesgue's measure, then $F'(u)$
is not boundedly invertible in $H$.
If $u\in C[0,1]$ vanishes even at one point $x_0$, then $F'(u)$ 
is not boundedly invertible in $H$.

We use the following iterative scheme 
\begin{equation}
\label{544eq11-51}
\begin{split}
u_{n+1} &= u_n - (F'(u_n)+a_n I)^{-1}(F(u_n)+a_nu_n - f_\delta),\\
u_0 &= 0,
\end{split}
\end{equation}
and stop iterations at $n:=n_\delta$ such that the following inequality holds
\begin{equation}
\label{544eq53}
\|F(u_{n_\delta}) - f_\delta\| <C \delta^\gamma,\quad
\|F(u_{n}) - f_\delta\|\ge C\delta^\gamma,\quad n<n_\delta ,\quad C>1,\quad \gamma \in(0,1).
\end{equation}
The existence of the stopping time $n_\delta$ is proved 
and the choice $u_0=0$ is also justified in this paper.
The drawback of the iterative scheme \eqref{544eq11-51} compared to the DSM 
in this paper is that
the solution $u_{n_\delta}$ may converge not to the minimal-norm
 solution to equation \eqref{eq1} but to another solution to this 
equation, if this equation has many solutions. 
There might be  other iterative schemes which are more efficient 
than scheme \eqref{544eq11-51}, but this scheme is simple 
and easy to implement.
 
Integrals of the form 
$\int_0^1 e^{-|x-y|}h(y)dy$ in \eqref{5441eq41} and \eqref{544eq44} are computed by 
using
the trapezoidal rule. The noisy function used in the test is
$$
f_\delta(x) = f(x) + \kappa f_{noise}(x),\quad \kappa>0.
$$
The noise level $\delta$ and the relative noise level are defined by 
the formulas:
$$
\delta = \kappa\| f_{noise}\|,\quad \delta_{rel}:=\frac{\delta}{\|f\|}.
$$
In the test $\kappa$ is computed in such a way that the relative noise level
$\delta_{rel}$ equals to some desired value, i.e.,
$$
\kappa = \frac{\delta}{\| f_{noise}\|}=\frac{\delta_{rel}\|f\|}{\| f_{noise}\|}.
$$
We have used the relative noise level as an input parameter in the test.

In the test we took $h=1$, $C =1.01$, and $\gamma = 0.99$.
The exact solution in the first test is $u(x)=1$, and the right-hand side
is $f=F(1)$.
 
It is proved that one can take 
$a_n=\frac{d}{1+n}$, and $d$ is
sufficiently large. However, in practice, if we choose $d$ too large,
then the method will use too many iterations before reaching the stopping
time $n_\delta$ in \eqref{544eq53}. This means that the computation time will
be large in this case. Since 
$$ 
\|F(V_\delta) - f_\delta\| = a(t)\|V_\delta\|, 
$$ 
and $\|V_\delta(t_\delta) - u_\delta(t_\delta)\|=O(a(t_\delta))$, 
we have 
$$
C\delta^\gamma =\|F(u_\delta(t_\delta)) - f_\delta\|\leq
a(t_\delta)\|V_\delta\|+ O(a(t_\delta)), 
$$
and we choose 
$$ 
d = C_0\delta^\gamma,\qquad C_0>0. 
$$ 
In the experiments our method works well with 
$C_0\in[3,10]$. In the
test we chose $a_n$ by the formula $a_n := C_0\frac{\delta^{0.99}}{n+1}$.
The number of nodal points, used in
computing integrals in \eqref{5441eq41} and \eqref{544eq44}, was $N = 50$. 
The accuracy of the solutions obtained
in the tests with $N=20$ and $N=30$ was
about the same as for $N = 50$.

Numerical results for various values of $\delta_{rel}$ are presented in Table~\ref{544table1}. 
In this experiment,
the noise function $f_{noise}$ is a vector with random entries normally 
distributed, with 
mean value 0 and variance 1. 
Table~\ref{544table1} shows that the iterative scheme yields good numerical results. 
\begin{table}[ht] 
\caption{Results when $C_0=4$ and $N=50$.}
\label{544table1}
\centering
\small
\begin{tabular}{|@{  }c@{\hspace{2mm}}
@{\hspace{2mm}}|c@{\hspace{2mm}}|c@{\hspace{2mm}}|c@{\hspace{2mm}}|c@{\hspace{2mm}}|
c@{\hspace{2mm}}|c@{\hspace{2mm}}|c@{\hspace{2mm}}r@{\hspace{2mm}}l@{}} 
\hline
$\delta_{rel}$       &0.05   &0.03   &0.02    &0.01    &0.003 &0.001\\
\hline
Number of iterations &28	 &29     &28     &29   &29    &29\\
\hline 
$\frac{\|u_{DSM} - u_{exact}\|}{\|u_{exact}\|}$&0.0770	 &0.0411   &0.0314    &0.0146   &0.0046    &0.0015\\
\hline
\end{tabular}
\end{table}


\end{document}